\documentclass[11pt]{article}

\usepackage{amssymb,amsmath,euscript,bbm,xcolor,graphicx,epstopdf}
\RequirePackage[numbers]{natbib}

\usepackage{amsmath,amssymb,amsthm,bm,euscript,bbm}
\usepackage{verbatim}
\usepackage{graphicx}
\usepackage{epstopdf}
\usepackage{color}
\usepackage{pstricks,pst-node,pst-plot}
\usepackage{rotating}
\usepackage{enumitem}
\newtheorem{proposition}{Proposition}[section]
\newtheorem{lemma}[proposition]{Lemma}
\newtheorem{theorem}[proposition]{Theorem}

\newtheorem{corollary}[proposition]{Corollary}

\def\la{\lambda}
\def\La{\Lambda}
\def\ep{\varepsilon}

\def\l{{\langle}}
\def\r{\rangle}

\def\R{{\mathbb R}}
\def\S{{\mathbb S}}

\def\E{{\mathbb E}}
\def\P{{\mathbb P}}

\makeatletter \@addtoreset{equation}{section} \makeatother
\newenvironment{remark}{%
    \vspace{0.3cm} \pagebreak [2]%
    \par%
    \refstepcounter{proposition}
    \noindent%
    {\bf Remark~\theproposition\  }}{\qed}%
\begin{document}

\title {The expected Euler characteristic approximation to excursion probabilities of smooth Gaussian random fields with general variance functions}
\author{Dan Cheng \\ Arizona State University}

\maketitle

\begin{abstract}
	Consider a centered smooth Gaussian random field $\{X(t), t\in T \}$ with a general (nonconstant) variance function. In this work, we demonstrate that as $u \to \infty$, the excursion probability $\P\{\sup_{t\in T} X(t) \geq u\}$ can be accurately approximated by $\E\{\chi(A_u)\}$ such that the error decays at a super-exponential rate. Here, $A_u = \{t\in T: X(t)\geq u\}$ represents the excursion set above $u$, and $\E\{\chi(A_u)\}$ is the expectation of its Euler characteristic $\chi(A_u)$. This result substantiates the expected Euler characteristic heuristic for a broad class of smooth Gaussian random fields with diverse covariance structures. In addition, we employ the Laplace method to derive explicit approximations to the excursion probabilities.
\end{abstract}



\section{Introduction}
Let $X = \{X(t),\, t\in T \}$ represent a real-valued Gaussian random field defined on the probability space $(\Omega, \mathcal{F}, \P)$, where $T$ denotes the parameter space. The study of excursion probabilities, denoted as $\P \{\sup_{t\in T} X(t) \geq u \}$, is a classical and fundamental problem in both probability and statistics. It finds extensive applications across numerous domains, including $p$-value computations, risk control and extreme event analysis, etc. 

In the field of statistics, excursion probabilities play a critical role in tasks such as controlling family-wise error rates \cite{TaylorW07,TaylorW08}, constructing confidence bands \cite{Sun:2001}, and detecting signals in noisy data \cite{Siegmund:1995,TaylorW07}. However, except for only a few examples, computing the exact values of these probabilities is almost  impossible. To address this challenge, many researchers have developed various methods for precise approximations of $\P \{\sup_{t\in T} X(t) \geq u \}$. These methods encompass techniques like the double sum method \cite{Piterbarg:1996}, the tube method \cite{Sun93} and the Rice method \cite{AzaisD02,AzaisW09}. For comprehensive theoretical insights and related applications, we refer readers to the survey by \citet{Adler00} and the monographs by \citet{Piterbarg:1996}, \citet{Adler:2007}, and \citet{AzaisW09}, as well as the references therein. 

In recent years, the expected Euler characteristic (EEC) method has emerged as a powerful tool for approximating excursion probabilities. This method, originating from the works of \citet{TTA05} and \citet{Adler:2007}, provides the following approximation:
\begin{equation}\label{Equation:mean EC approximation error}
	\P \bigg\{\sup_{t\in T} X(t) \geq u \bigg\} = \E\{\chi(A_u)\} + {\rm error}, \quad  {\rm as} \ u\to \infty,
\end{equation}
where $\chi(A_u)$ represents the Euler characteristic of the excursion set $A_u = \{t\in T: X(t)\geq u\}$. This approximation \eqref{Equation:mean EC approximation error} is highly elegant and accurate, primarily due to the fact that the principle term $\E\{\chi(A_u)\}$ is computable and the error term decays exponentially faster than the major component. However, it is essential to note that this method assumes a Gaussian field with constant variance, limiting its applicability in various scenarios.

In this paper, we extend the EEC method to accommodate smooth Gaussian random fields with general (nonconstant) variance functions. Our main objective is to demonstrate that the EEC approximation \eqref{Equation:mean EC approximation error} remains valid under these conditions, with the error term exhibiting super-exponential decay. For a precise description of our findings, please refer to Theorem \ref{Thm:MEC approximation je} below. Our derived approximation result shows that the maximum variance  of  $X(t)$, denoted by $\sigma_T^2$ (see \eqref{eq:R} below), plays a pivotal role in both $\E\{\chi(A_u)\}$ and the super-exponentially small error. In our analysis, we observe that the points where $\sigma_T^2$ is attained make the most substantial contributions to  $\E\{\chi(A_u)\}$. Building on this observation, we establish two simpler approximations: one in Theorem \ref{Thm:MEC approximation je2}, which incorporates boundary conditions on nonzero derivatives of the variance function over points where $\sigma_T^2$ is attained, and another in Theorem \ref{Thm:MEC approximation}, assuming only a single point attains $\sigma_T^2$.

In general, the EEC approximation can be expressed as an integral using the Kac-Rice formula, as outlined in \eqref{eq:EEC} in Theorem \ref{Thm:MEC approximation je}. While \cite{TTA05,Adler:2007} provided an elegant expression for $\E\{\chi(A_u)\}$ termed the Gaussian kinematic formula, this expression heavily relies on the assumption of unit variance, which simplifies the calculation. In our case, where the variance function of $X(t)$ varies across $T$, deriving an explicit expression for $\E\{\chi(A_u)\}$ becomes challenging. Instead, we apply the Laplace method to extract the term with the leading order of $u$ from the integral, leaving a remaining error that is $\E\{\chi(A_u)\}o(1/u)$. For a more detailed explanation, we offer specific calculations in Sections \ref{sec:unique-max} and \ref{sec:example}. To intuitively grasp the EEC approximation, one can roughly consider the major term as $g(u)e^{-u^2/(2\sigma_T^2)}$, while the error term diminishes as $o(e^{-u^2/(2\sigma_T^2) - \alpha u^2})$, where $g(u)$ is a polynomial in $u$, and $\alpha>0$ is a constant.

The structure of this paper is as follows: We begin by introducing the notations and assumptions in Section \ref{sec:notation}. In Section \ref{sec:main}, we present our main results, including Theorems \ref{Thm:MEC approximation je}, \ref{Thm:MEC approximation je2}, and \ref{Thm:MEC approximation}. To understand our approach, we outline the main ideas in Section \ref{sec:sketch} and delve into the analysis of super-exponentially small errors in Sections \ref{sec:small} and \ref{sec:diff}. Finally, we provide the proofs of our main results in Section \ref{sec:proof}. In Section \ref{sec:unique-max}, we apply the Laplace method to derive explicit approximations (Theorems \ref{thm:unique-max-boundary} and \ref{thm:unique-max}) for cases where a unique maximum point of the variance is present. In Section \ref{sec:example}, we demonstrate several examples that illustrate the evaluation of EEC and the subsequent approximation of excursion probabilities.

\section{Notations and assumptions}\label{sec:notation}

Let $\{X(t),\, t\in T\}$ be a real-valued and centered Gaussian random field, where $T$ is a compact rectangle in $\R^N$. We define
\begin{equation}\label{eq:R}
\begin{split}
\nu(t) = \sigma_t^2 = {\rm Var}(X(t)) \quad \text{\rm and }  \quad \sup_{t \in T} \nu(t) =\sigma_T^2.
\end{split}
\end{equation}
Here, $\nu(\cdot)$ represents the variance function of the field and $\sigma_T^2$ is the maximum variance over $T$. For a function $f(\cdot) \in C^2(\R^N)$ and $t\in \R^N$, we introduce the following notations on derivatives:
\begin{equation}\label{Eq:notatoin-diff}
\begin{split}
f_i (t)&=\frac{\partial f(t)}{\partial t_i}, \quad f_{ij}(t)=\frac{\partial^2 f(t)}{\partial t_i\partial t_j}, \quad \forall i, j=1, \ldots, N;\\
\nabla f(t) &= (f_1(t), \ldots , f_N(t))^{T}, \quad \nabla^2 f(t) = \left(f_{ij}(t)\right)_{ i, j = 1, \ldots, N}.
\end{split}
\end{equation}
Let $B \prec 0$ (negative definite) and $B \preceq 0$ (negative semi-definite) denote that a symmetric matrix $B$ has all negative or nonpositive eigenvalues, respectively. Additionally, we use ${\rm Cov}(\xi_1, \xi_2)$ and ${\rm Corr}(\xi_1, \xi_2)$ to represent the covariance and correlation between two random variables $\xi_1$ and $\xi_2$. The density of the standard Normal distribution is denoted as $\phi(x)$, and its tail probability is $\Psi(x) = \int_x^\infty \phi(y)dy$. Let $\S^j$ be the $j$-dimensional unit sphere.

Consider the domain $T=\prod^N_{i=1}[a_i, b_i]$, where $-\infty< a_i<b_i<\infty$. We draw from the notation established by Adler and Taylor in \cite{Adler:2007} to demonstrate that $T$ can be decomposed into the union of its interior and lower-dimensional faces. This decomposition forms the basis for calculating the Euler characteristic of the excursion set $A_u$, as elaborated in Section \ref{sec:main}.

Each face $K$ of dimension $k$ is defined by fixing a subset $\tau(K) \subset \{1, \ldots, N\}$ of size $k$ and a subset $\varepsilon(K) 
= \{\varepsilon_j, j\notin \tau(K)\} \subset \{0, 1\}^{N-k}$ of size $N-k$ so that
\begin{equation*}
\begin{split}
K= \{ t=(t_1, \ldots, t_N) \in T: \, &a_j< t_j <b_j  \ {\rm if} \ j\in \tau(K), \\
&t_j = (1-\ep_j)a_j + \ep_{j}b_{j}  \ {\rm if} \ j\notin \tau(K)  \}.
\end{split}
\end{equation*}
Denote by $\partial_k T$ the collection of all $k$-dimensional faces in $T$. The interior of $T$ is designated as $\overset{\circ}{T}=\partial_N T$, while the boundary of $T$ is formulated as $\partial T = \cup^{N-1}_{k=0}\cup _{K\in \partial_k T} K$. This allows us to partition $T$ in the following manner:
\begin{equation*}
	T= \bigcup_{k=0}^N \partial_k T = \bigcup_{k=0}^N\bigcup_{K\in \partial_k T}K.
\end{equation*}
For each $t\in T$, let
\begin{equation}\label{eq:Sigma}
\begin{split}
\nabla X_{|K}(t) &= (X_{i_1} (t), \ldots, X_{i_k} (t))^T_{i_1, \ldots, i_k \in \tau(K)}, \quad \nabla^2 X_{|K}(t) = (X_{mn}(t))_{m, n \in \tau(K)},\\
		\Sigma(t)&= \E\{X(t)\nabla^2X(t)\} = (\E\{X(t)X_{ij}(t)\})_{1\leq i,j\leq N}, \\ \Sigma_K (t)&=\E\{X(t)\nabla^2X_{|K}(t)\} = (\E\{X(t)X_{ij}(t)\})_{i,j\in \tau(K)},\\
		\La(t) &= {\rm Cov}(\nabla X(t)) = (\E\{X_i(t)X_j(t)\})_{1\leq i,j\leq N},\\
		\La_K(t) &= {\rm Cov}(\nabla X_{|K}(t))=(\E\{X_i(t)X_j(t)\})_{i,j\in \tau(K)}.
\end{split}
\end{equation}
For each $K\in \partial_k T$, we define the \emph{number of extended outward maxima
above $u$ on face $K$} as
\begin{equation*}
\begin{split}
M_u^E (K) & := \# \{ t\in K: X(t)\geq u, \nabla X_{|K}(t)=0, \nabla^2 X_{|K}(t)\prec 0, \varepsilon^*_jX_j(t) \geq 0, \forall j\notin \tau(K) \},
\end{split}
\end{equation*}
where $\ep^*_j=2\ep_j-1$, and define the \emph{number of local maxima above $u$ on face $K$} as
\begin{equation*}
\begin{split}
M_u (K) & := \# \{ t\in K: X(t)\geq u, \nabla X_{|K}(t)=0, \nabla^2 X_{|K}(t)\prec 0\}.
\end{split}
\end{equation*}
 Clearly, $M_u^E (X,K) \le M_u (X,K)$. 

For each $t\in T$ with $\nu(t)=\sigma_T^2$, we define the index set $\mathcal{I}(t) = \{ \ell:  \nu_\ell(t)=0\}$ representing the directions along which the partial derivatives of $\nu(t)$ vanish. If $t\in K\in \partial_k T$ with $\nu(t)=\sigma_T^2$, then we have $\tau(K) \subset \mathcal{I}(t)$ since $\nu_\ell(t) =0$ for all $\ell \in \tau(K)$. It is worth noting that since $\nu_i(t) = 2\E\{X_{i}(t)X(t)\}$, we can also express this index set as $\mathcal{I}(t) = \{ \ell: \E\{X(t)X_\ell(s)\}=0\}$. 

Our analytical framework relies on the following conditions for smoothness ({\bf H}1) and regularity ({\bf H}2), in addition to curvature conditions ({\bf H}3) or $({\bf H}3')$.
\begin{itemize}
	\item[({\bf H}1)] $X\in C^2(\R^N)$ almost surely and the second derivatives satisfy the
	\emph{uniform mean-square H\"older condition}: there exist constants $C, \delta>0$ such that
	\begin{equation*}
	\begin{split}
	\E(X_{ij}(t)-X_{ij}(t'))^2 &\leq C \|t-t'\|^{2\delta}, \quad \forall t,t'\in T,\ i, j= 1, \ldots, N.
	\end{split}
	\end{equation*}
	
	\item[({\bf H}2)] For every pair $(t, t')\in T^2$ with $t\neq t'$, the Gaussian vector
	\begin{equation*}
	\begin{split}
	\big(&X(t), \nabla X(t), X_{ij}(t), X(t'), \nabla X(t'), X_{ij}(t'), 1\leq i\leq j\leq N\big)
	\end{split}
	\end{equation*}
	is  non-degenerate.
	
	\item[({\bf H}3)] For every $t\in K\in \partial_k T$, $0\le k\le N-2$, such that $\nu(t)=\sigma_T^2$ and $\mathcal{I}(t)$ contains at least two indices, we have
	\begin{equation}\label{eq:Hessian_Sigma}
		(\E\{X(t)X_{ij}(t)\})_{i,j\in \mathcal{I}(t) }\prec 0.
	\end{equation} 
	
	\item[$({\bf H}3')$] For every $t\in K\in \partial_k T$, $0\le k\le N-2$, such that $\nu(t)=\sigma_T^2$ and $\mathcal{I}(t)$ contains at least two indices, we have
	\begin{equation}\label{eq:Hessian_I}
	\left(\nu_{ij}(t)\right)_{i,j\in \mathcal{I}(t) }\preceq 0.
\end{equation}
\end{itemize}

Conditions ({\bf H}3) and $({\bf H}3')$ involve the behavior of the variance function $\nu(t)$ at critical points, and they are closely related, as shown in Proposition \ref{prop:H3} below. Here we provide some additional insights into $({\bf H}3')$. Despite its initially technical appearance,  $({\bf H}3')$ is in fact a mild condition that specifically applies to lower-dimensional boundary points $t$ where $\nu(t)=\sigma_T^2$. In essence, it indicates that the variance function should possess a negative semi-definite Hessian matrix at these boundary critical points where $\nu(t)=\sigma_T^2$ while concurrently exhibiting at least two zero partial derivatives. 

For example, in the 1D case, since $\mathcal{I}(t)$ contains at most one index, there is no need to check $({\bf H}3')$. Similarly, in the 2D case, we only need to check $({\bf H}3')$ or \eqref{eq:Hessian_I} when $\sigma_T^2$ is achieved at corner points $t\in \partial_0 T$ with $\mathcal{I}(t)=\{1, 2\}$. Moreover, if the variance function $\nu(t)$ demonstrates strict monotonicity in all directions across $\R^N$, then $\mathcal{I}(t) = \emptyset$ and there is no need to verify $({\bf H}3')$.

\begin{proposition}\label{prop:H3}
	The condition $({\bf H}3')$ implies $({\bf H}3)$. In addition, $({\bf H}3)$ implies that
	\begin{equation}\label{eq:prop:H3}
		(\E\{X(t)X_{ij}(t)\})_{i,j\in \mathcal{I}(t) }\prec 0, \quad \forall t\in T \text{ with }  \nu(t)=\sigma_T^2.
	\end{equation}
\end{proposition}
\begin{proof}
	Taking the second derivative on both sides of $\nu(t)=\E\{X(t)^2\}$, we obtain $\nu_{ij}(t)/2 = \E\{X(t)X_{ij}(t)\} + \E\{X_i(t)X_j(t)\}$, implying
	\begin{equation}\label{eq:negdef}
	(\E\{X(t)X_{ij}(t)\})_{i, j \in \mathcal{I}(t)} = \frac{1}{2}(\nu_{ij}(t))_{i, j \in\mathcal{I}(t)} - (\E\{X_i(t)X_j(t)\})_{i, j \in \mathcal{I}(t)}.
	\end{equation}
	Note that, as a covariance matrix, $(\E\{X_i(t)X_j(t)\})_{i, j \in \mathcal{I}(t)}$ is positive definite by $({\bf H}2)$. Therefore, \eqref{eq:Hessian_I} implies \eqref{eq:Hessian_Sigma}, or equivalently $({\bf H}3')$ implies $({\bf H}3)$. 
	
	Next we demonstrate that $({\bf H}3)$ implies \eqref{eq:prop:H3}. It suffices to show \eqref{eq:Hessian_Sigma} for $k= N-1$ and $k=N$, and for the case that $\mathcal{I}(t)$ contains at most one index, which complement those cases in $({\bf H}3)$.
	
	(i) If $k = N$, then $t$ becomes a maximum point of $\nu$ within the interior of $T$ and $\mathcal{I}(t) = \tau(K) =\{1, \cdots, N\}$, implying \eqref{eq:Hessian_I}, and hence \eqref{eq:Hessian_Sigma} holds by \eqref{eq:negdef}.
	
	(ii) For $k=N-1$, we consider two scenarios. If $\mathcal{I}(t) = \tau(K)$, then $t$ becomes a maximum point of $\nu$ restricted on $K$, hence \eqref{eq:Hessian_Sigma} is satisfied as discussed above. If $\mathcal{I}(t) =\{1, \cdots, N\}$, then it follows from Taylor's formula that
	\begin{equation*}
		\nu (t') =\nu(t) + (t'-t)^T \nabla^2 \nu(t) (t'-t)+o(\|t'-t\|^2), \quad t'\in T.
	\end{equation*}
	Notice that $\{(t'-t)/\|t'-t\|: t'\in T\}$ contains all directions in $\R^N$ 
	since $t\in K \in \partial_{N-1}T$, together with the fact $\nu(t)=\sigma_T^2$, we see that $\nabla^2 \nu(t)$ cannot have any positive eigenvalue, thus \eqref{eq:Hessian_I} and hence \eqref{eq:Hessian_Sigma} hold.
	
	(iii) Finally, it's evident from the 1D Taylor's formula that \eqref{eq:Hessian_I} is valid when $\mathcal{I}(t)$ contains only one index. 
\end{proof}

	The condition \eqref{eq:prop:H3} established in Proposition \ref{prop:H3} serves as the fundamental requirement for our main results, as demonstrated in Theorems \ref{Thm:MEC approximation je}, \ref{Thm:MEC approximation je2} and \ref{Thm:MEC approximation} below. As seen from Proposition \ref{prop:H3}, we can simplify \eqref{eq:prop:H3} to condition $({\bf H}3)$. Thus our main results will be presented under the assumption of condition $({\bf H}3)$.
	
	Furthermore, it is worth highlighting that, in practical applications, verifying $({\bf H}3')$ can often be a more straightforward process. This condition directly pertains to the variance function $\nu(t)$, making it easier to assess. Thus, Proposition \ref{prop:H3} provides the flexibility to check $({\bf H}3')$ instead of $({\bf H}3)$. This insight simplifies the verification procedure, enhancing the practical applicability of our results.


\section{Main results}\label{sec:main}
Here, we will present our main results Theorems \ref{Thm:MEC approximation je}, \ref{Thm:MEC approximation je2} and \ref{Thm:MEC approximation}, whose proofs are given in Section \ref{sec:proof}. Define the \emph{number of extended outward critical points of index $i$
	above level $u$ on face $K$} be
\begin{equation*}
\begin{split}
\mu_i(K) & := \# \{ t\in K: X(t)\geq u, \nabla X_{|K}(t)=0, \text{index} (\nabla^2 X_{|K}(t))=i, \\
& \qquad \qquad \qquad \varepsilon^*_jX_j(t) \geq 0 \ {\rm for \ all}\ j\notin \tau(K) \}.
\end{split}
\end{equation*}
Recall that  $\ep^*_j=2\ep_j-1$ and the index of a matrix is defined as the number of its negative
eigenvalues. It is evident to observe that $\mu_N(K)=M_u^E(K)$. It follows from ({\bf H}1), ({\bf H}2) and the Morse theorem (see  Corollary 9.3.5 or 
pages 211--212 in \citet{Adler:2007}) that the Euler characteristic of the excursion set $A_u$ can be represented as
\begin{equation}\label{eq:Euler}
\begin{split}
\chi(A_u)&= \sum^N_{k=0}\sum_{K\in \partial_k T}(-1)^k\sum^k_{i=0} (-1)^i \mu_i(K).
\end{split}
\end{equation}
Now we state the following general result on the EEC approximation for the excursion probability.
\begin{theorem}\label{Thm:MEC approximation je} 
Let $\{X(t),\, t\in T\}$ be a centered Gaussian random field satisfying $({\bf H}1)$, $({\bf H}2)$ and $({\bf H}3)$. Then there exists a constant
$\alpha>0$ such that as $u\to \infty$,
\begin{equation}\label{eq:EEC}
	\begin{split}
	&\quad \P \left\{\sup_{t\in T} X(t) \geq u \right\}\\
	&=\sum^N_{k=0}\sum_{K\in \partial_k T}(-1)^k\int_K \E\big\{{\rm det}\nabla^2 X_{|K}(t)\mathbbm{1}_{\{X(t)\geq u, \ \varepsilon^*_\ell X_\ell(t) \geq 0 \ {\rm for \ all}\ \ell\notin \tau(K)\}} \big| \nabla X_{|K}(t)=0\big\}\\
	&\quad \times p_{\nabla X_{|K}(t)}(0)dt   +o\left( \exp \left\{ -\frac{u^2}{2\sigma_T^2} -\alpha u^2 \right\}\right)\\
	&= \E\{\chi(A_u)\}+o\left( \exp \left\{ -\frac{u^2}{2\sigma_T^2} -\alpha u^2 \right\}\right).
	\end{split}
	\end{equation}
\end{theorem}

In general, computing the EEC approximation $\E\{\chi(A_u)\}$ is a challenging task because it involves conditional expectations over the joint covariance of the Gaussian field and its Hessian, given zero gradient, which vary across $T$. However, one can apply the Laplace method to extract the term with the largest order of $u$ from $\E\{\chi(A_u)\}$ such that the remaining error is $o(1/u)\E\{\chi(A_u)\}$. Examples demonstrating the Laplace method are presented in Section \ref{sec:example}.
	
It is important to note that in the expression \eqref{eq:EEC}, when $k=0$, all terms involving $\nabla X_{|K}(t)$ and $\nabla^2 X_{|K}(t)$ vanish. Consequently, if $k=0$, we treat the integral in \eqref{eq:EEC} as the usual Gaussian tail probabilities. This notation is also adopted in the results presented in Theorems \ref{Thm:MEC approximation je2} and \ref{Thm:MEC approximation} below.

The proof of Theorem \ref{Thm:MEC approximation je} reveals that points where the maximum variance $\sigma_T^2$ is attained make the most significant contribution to $\E\{\chi(A_u)\}$. Therefore, in many cases, the general EEC approximation $\E\{\chi(A_u)\}$ can be simplified. The following result is based on the boundary condition \eqref{Eq:boundary} and is applicable at boundary points where nonzero partial derivatives of the variance function occur when $\sigma_T^2$ is reached.
\begin{theorem}\label{Thm:MEC approximation je2} Let $\{X(t),\, t\in T\}$ be a centered Gaussian random field satisfying $({\bf H}1)$, $({\bf H}2)$ and the following boundary condition 
\begin{equation}\label{Eq:boundary}
	\begin{split}
		\Big\{t\in J:\, \nu(t)=\sigma_T^2, \prod_{i\notin \tau(J)}\nu_i(t)=0\Big\} = \emptyset, \quad \forall \text{ face } J\subset T.
	\end{split}
\end{equation} 
Then there exists a constant $\alpha>0$ such that as $u\to \infty$,
	\begin{equation*}
	\begin{split}
	\P \left\{\sup_{t\in T} X(t) \geq u \right\}
	&=\sum^N_{k=0}\sum_{K\in \partial_k T}(-1)^{k}\int_K  
	\E\big\{{\rm det}\nabla^2 X_{|K}(t) \mathbbm{1}_{\{X(t)\geq u\}}\big|\nabla X_{|K}(t)=0\big\}\\
	&\quad   \times p_{\nabla X_{|K}(t)}(0)dt  +o\left( \exp \left\{ -\frac{u^2}{2\sigma_T^2} -\alpha u^2 \right\}\right).
	\end{split}
	\end{equation*}
\end{theorem}

In other words, the boundary condition \eqref{Eq:boundary} indicates that, for any point $t\in J$ attaining the maximum variance $\sigma_T^2$, there must be $\nu_i(t)\neq 0$ for all $i\notin \tau(J)$. In particular, as an important property, we observe that \eqref{Eq:boundary} implies the condition $({\bf H}3')$ and hence $({\bf H}3)$. The following result provides an asymptotic approximation for the special case where the variance function attains its maximum $\sigma_T^2$ only at a unique point.

\begin{theorem}\label{Thm:MEC approximation} Let $\{X(t),\, t\in T\}$ be a centered Gaussian random field satisfying $({\bf H}1)$, $({\bf H}2)$ and $({\bf H}3)$. Suppose $\nu(t)$ attains its maximum $\sigma_T^2$ only at a single point $t^*\in K$, where $K\in \partial_k T$ with $k\ge 0$. Then there exists a constant $\alpha>0$ such that as $u\to \infty$,
	\begin{equation*}
		\begin{split}
			&\quad \P \left\{\sup_{t\in T} X(t) \geq u \right\}\\
			&=\sum_{J}(-1)^{{\rm dim}(J)}\int_J \E\big\{{\rm det}\nabla^2 X_{|J}(t)\mathbbm{1}_{\{X(t)\geq u, \ \varepsilon^*_\ell X_\ell(t) \geq 0 \ {\rm for \ all}\ \ell\in \mathcal{I}(t^*)\setminus \tau(J)\}} \big| \nabla X_{|J}(t)=0\big\}\\
			&\quad \times p_{\nabla X_{|J}(t)}(0)dt +o\left( \exp \left\{ -\frac{u^2}{2\sigma_T^2} -\alpha u^2 \right\}\right),
		\end{split}
	\end{equation*}
where the sum is taken over all faces $J$ of $T$ such that $t^*\in \bar{J}$ and $\tau(J)\subset \mathcal{I}(t^*)$.
\end{theorem}

Employing the Laplace method, we will provide refined explicit approximation results in Section \ref{sec:unique-max} under the assumptions in Theorem \ref{Thm:MEC approximation}. Furthermore, we demonstrate several examples that illustrate the evaluation of approximating excursion probabilities in Section \ref{sec:example}.

\section{Outline of the proofs}\label{sec:sketch}

Here we show the main idea for proving the main results above. Let $f$ be a smooth real-valued function, then $\sup_{t\in T} f(t) \geq u$ if and only if there exists at least one extended outward local maximum above $u$ on some face of $T$. Thus, under conditions $({\bf H}1)$ and $({\bf H}2)$, the following 
relation holds for each $u\in \R$:
\begin{equation}\label{Eq:maxima-faces}
\begin{split}
 \left\{\sup_{t\in T} X(t) \geq u \right\} = \bigcup_{k=0}^N \bigcup_{K\in \partial_k T}
\{M_u^E (K) \geq 1\} \quad {\rm a.s.}
\end{split}
\end{equation}
This implies that the probability of the supremum of the Gaussian random field exceeding $u$ is equal to the probability that there exists at least one extended outward local maximum above $u$ on some face $K$ of $T$. Therefore, we obtain the following upper bound for the excursion probability:
\begin{equation}\label{Ineq:upperbound je}
\begin{split}
 \P\left\{\sup_{t\in T} X(t) \geq u\right\}\leq \sum_{k=0}^N\sum_{K\in \partial_k T} \P\{M_u^E (K) \geq 1\}\leq \sum_{k=0}^N\sum_{K\in \partial_k T} \E \{M_u^E (K) \}.
\end{split}
\end{equation}
On the other hand, notice that
\begin{equation*}
\begin{split}
\E\{M_u^E (K) \} - \P\{M_u^E (K) \geq 1\} 
&= \sum_{i=1}^\infty (i-1)\P\{M_u^E (K)=i\}\\
& \leq \sum_{i=1}^\infty i(i-1)\P\{M_u^E (K)=i\}= \E \{M_u^E (K)[M_u^E (K)-1] \} 
\end{split}
\end{equation*}
and
\begin{equation*}
\begin{split}
\P\{M_u^E (K) \geq 1, M_u^E (K') \geq 1\}\le \E\{M_u^E (K) M_u^E (K') \}.
\end{split}
\end{equation*}
Applying the Bonferroni inequality to \eqref{Eq:maxima-faces} and combining these two inequalities, we 
obtain the following lower bound for the excursion probability:
\begin{equation}\label{Ineq:lowerbound je}
\begin{split}
&\quad \P \left\{\sup_{t\in T} X(t) \geq u \right\} \\
&\ge \sum_{k=0}^N\sum_{K\in \partial_k T} \P\{M_u^E (K) \geq 1\} - \sum_{K\neq K'} \P\{M_u^E (K) \geq 1, M_u^E (K') \geq 1\}  \\
&\ge \sum_{k=0}^N\sum_{K\in \partial_k T} \left(\E\{M_u^E (K) \} - \E \{M_u^E (K)[M_u^E (K)-1] \} \right)  - \sum_{K\neq K'} \E\{M_u^E (K) M_u^E (K') \},
\end{split}
\end{equation}
where the last sum is taken over all possible pairs of different faces $(K, K')$. 
\begin{remark}\label{remark:M_u}
	Note that, following the same arguments above, we have that the expectations on the number of extended outward maxima $M_u^E(\cdot)$ in both \eqref{Ineq:upperbound je} and \eqref{Ineq:lowerbound je} can be replaced by the expectations on the number of local maxima $M_u(\cdot)$.
\end{remark}

We call a function $h(u)$ \emph{super-exponentially small} [when compared with the excursion probability $\P \{\sup_{t\in T} X(t) \geq u \}$ or $\E\{\chi(A_u)\}$], if there exists a constant $\alpha >0$ such that $h(u) = o(e^{- u^2/{(2\sigma_T^2)-\alpha u^2}})$ as $u \to \infty$. The main idea for proving the EEC approximation Theorem \ref{Thm:MEC approximation je} consists of the following two steps: (i) show that, except for the upper bound in (\ref{Ineq:upperbound je}), all terms in the lower bound in (\ref{Ineq:lowerbound je}) are super-exponentially small; and (ii) demonstrate that the difference between the upper bound in (\ref{Ineq:upperbound je}) and $\E\{\chi(A_u)\}$ is also super-exponentially small. The proofs for Theorems \ref{Thm:MEC approximation je2} and \ref{Thm:MEC approximation} follow the same ideas, aiming to establish super-exponential smallness for the terms involved in the lower bounds, as well as for the difference between the upper bound and EEC.

\section{Estimation of super-exponential smallness for terms in the lower bound}\label{sec:small}

\subsection{Factorial moments}
We first state the following result, which is a modified version (restricted on a face $K$) of Lemma 4 in \citet{Piterbarg96}, characterizing the decaying rate for factorial moments of the number of critical points exceeding a high level for Gaussian fields.
\begin{lemma} \label{lemma:Piterbarg}
	Assume $({\bf H}1)$ and $({\bf H}2)$.
	Then there exists a positive constant $C$ such that for any $\varepsilon >0$ one can find a number $\varepsilon_1 >0$ such that for any $K\in\partial_k T$,
	\begin{equation}
		\begin{split}
			\E \{M_u (K)(M_u(K)- 1)\}\leq C u^{2k+1} \exp\bigg\{-\frac{u^2}{2\beta_K^2 +\varepsilon}\bigg\} + C u^{4k+2}\exp\bigg\{-\frac{u^2}{2\sigma_K^2 -\ep_1}\bigg\},
		\end{split}
	\end{equation}
	where
	$$\beta_K^2 = \sup_{t\in K} \sup_{e\in \S^{k-1}} {\rm Var} (X(t)|\nabla X_{|K}(t), \nabla^2 X_{|K}(t)e), \quad \sigma_K^2=\sup_{t\in K} {\rm Var} (X(t)).$$
\end{lemma}

The following result shows that the factorial moments in \eqref{Ineq:lowerbound je} are super-exponentially small under our assumptions.
\begin{proposition}\label{Lem:factorial moments je} Let $\{X(t),\, t\in T\}$ be a centered Gaussian random field satisfying $({\bf H}1)$, $({\bf H}2)$ and $({\bf H}3)$. Then there exists $\alpha>0$ such that as $u\to \infty$,
	\begin{equation}
		\sum_{k=0}^N\sum_{K\in\partial_k T} \E \{M_u (K)(M_u (K) -1) \} = o\left(e^{ - u^2/(2\sigma_T^2) -\alpha u^2 }\right).
	\end{equation}
\end{proposition}

\begin{proof} Due to Lemma \ref{lemma:Piterbarg}, it suffices to show that for each $K\in\partial_k T$,  $\beta_K^2<\sigma_T^2$, which is equivalent to ${\rm Var} (X(t)|\nabla X_{|K}(t), \nabla^2 X_{|K}(t)e) <\sigma_T^2$ for all $t \in \bar{K} = K \cup \partial K$ and $e\in \S^{k-1}$. Suppose ${\rm Var} (X(t)|\nabla X_{|K}(t), \nabla^2 X_{|K}(t)e) =\sigma_T^2$ for some $t \in K$, then
	\begin{equation*}
		\sigma_T^2={\rm Var} (X(t)|\nabla X_{|K}(t), \nabla^2 X_{|K}(t)e) \leq {\rm Var} (X(t)|\nabla^2 X_{|K}(t)e) \leq {\rm Var} (X(t)) \leq \sigma_T^2.
	\end{equation*}
	Note that
	\begin{equation*}
		\begin{split}
			\quad {\rm Var} (X(t)|\nabla^2 X_{|K}(t)e) = {\rm Var} (X(t))\Leftrightarrow \E \{X(t)(\nabla^2 X_{|K}(t)e)\}=0 \Leftrightarrow \Sigma_K(t)e=0.
		\end{split}
	\end{equation*}
	But $t$ is a point with $\nu(t)=\sigma_T^2$, thus $\Sigma_K(t)\prec 0$ by Proposition \ref{prop:H3}, implying $\Sigma_K(t)e \neq 0$ for all $e\in \S^{k-1}$ and causing a contradiction.
	
	On the other hand, suppose ${\rm Var} (X(t)|\nabla X_{|K}(t), \nabla^2 X_{|K}(t)e) =\sigma_T^2$ for some $t \in \partial K$, then ${\rm Var} (X(t)|\nabla X_{|K}(t)) =\sigma_T^2$ and hence $\nu_i(t)=0$ for all $i\in \tau (K)$, implying $\Sigma_K(t)\prec 0$ by Proposition \ref{prop:H3}. Similarly to the previous arguments, this will lead to a contradiction. The proof is completed.
\end{proof}

\subsection{Non-adjacent faces}

For two sets $D, D' \subset \R^N$, let $d(D,D')=\inf\{\|t-t'\|: t\in D, t'\in D'\}$ denote their distance. The following 
result demonstrates that the last two sums involving the joint moment of two non-adjacent faces in \eqref{Ineq:lowerbound je} 
are super-exponentially small.
\begin{proposition}\label{Lem:cross terms disjoint sets} Let $\{X(t),\, t\in T\}$ be a centered Gaussian random field satisfying $({\bf H}1)$ and $({\bf H}2)$. Then there exists $\alpha>0$ such that as $u\to \infty$,
	\begin{equation}\label{Eq:crossing terms disjoint sets}
	\begin{split}
	\E\{M_u (K) M_u (K')\} &= o\left( \exp \left\{ -\frac{u^2}{2\sigma_T^2} -\alpha u^2 \right\}\right),
	\end{split}
	\end{equation}
	where $K$ and $K'$ are different faces of $T$ with $d(K,K')>0$.
\end{proposition}
\begin{proof} 
Consider first the case where ${\rm dim}(K) =k \geq 1$ and ${\rm dim}(K') =k' \geq 1$. By the Kac-Rice formula for high moments \cite{Adler:2007}, we have 
	\begin{equation}\label{Eq:disjoint faces}
	\begin{split}
	&\quad \E\{M_u (K) M_u (K') \}\\
	&= \int_{K} dt \int_{K'} dt'\, \E\big \{ |{\rm det} \nabla^2 X_{|K}(t) | |{\rm det} \nabla^2 X_{|K'}(t') | \mathbbm{1}_{\{X(t)\geq u, X(t')\geq u\}}\\
	&\quad \times  \mathbbm{1}_{\{\nabla^2 X_{|K}(t) \prec 0, \, \nabla^2 X_{|K'}(t') \prec 0\}} \big|\nabla X_{|K}(t)=0, \nabla X_{|K'}(t')=0 \big\}p_{\nabla X_{|K}(t), \nabla X_{|K'} (t')}(0,0) \\
	&\leq \int_{K} dt \int_{K'} dt' \int_u^\infty dx \int_u^\infty dx' \, p_{X(t), X(t')}(x,x')p_{\nabla X_{|K}(t), \nabla X_{|K'} (t')}(0,0|X(t)=x, X(t')=x')\\
	&\quad \times \E\big \{ |{\rm det} \nabla^2 X_{|K}(t) | |{\rm det} \nabla^2 X_{|K'}(t') |\big | X(t)=x, X(t')=x', \nabla X_{|K}(t)=0, \nabla X_{|K'}(t')=0 \} .
	\end{split}
	\end{equation}
	Notice that the following two inequalities hold: for constants $a_{i_1}$ and $b_{i_2}$,
	\begin{equation*}
	\begin{split}
	&\prod_{i_1=1}^k |a_{i_1}| \prod_{i_2=1}^{k'} |b_{i_2}| \leq \frac{\sum_{i_1=1}^k |a_{i_1}|^{k+k'} 
	+ \sum_{i_2=1}^{k'} |b_{i_2}|^{k+k'}}{k+k'};
	\end{split}
	\end{equation*}
	and for any Gaussian variable $\xi$ and positive integer $m$, by Jensen's inequality,
	\begin{equation*}
	\begin{split}
	\E |\xi|^m \leq \E (|\E\xi|+|\xi-\E\xi|)^m &\leq 2^{m-1} (|\E\xi|^m + \E |\xi-\E\xi|^m)= 2^{m-1} (|\E\xi|^m + B_m({\rm Var}(\xi))^{m/2}),
	\end{split}
	\end{equation*}
	where $B_m$ is some constant depending only on $m$. Combining these two inequalities with the well-known conditional 
	formula for Gaussian variables, we obtain that there exist positive constants $C_1$ and $N_1$ such that for sufficiently large $x$ and $x'$,
	\begin{equation}\label{Eq:disjoint faces 2}
	\begin{split}
	\sup_{t\in K, t'\in K'}&\E\big \{ |{\rm det} \nabla^2 X_{|K}(t) | |{\rm det} \nabla^2 X_{|K'}(t') | \big| X(t)=x, X(t')=x',\nabla X_{|K}(t)=0, \nabla X_{|K'}(t')=0 \big\}\\
	 & \leq C_1+(xx')^{N_1}.
	\end{split}
	\end{equation}
	Further, there exists $C_2>0$ such that
	\begin{equation}\label{Eq:disjoint faces 3}
	\begin{split}
	&\quad \sup_{t\in K, t'\in K'} p_{\nabla X_{|K}(t), \nabla X_{|K'} (t')}(0,0|X(t)=x, X(t')=x')\\
	& \leq  \sup_{t\in K, t'\in K'}(2\pi)^{-(k+k')/2}[{\rm det Cov} (\nabla X_{|K}(t), \nabla X_{|K'} (t') | X(t)=x, X(t')=x')]^{-1/2} \\
	& \leq  C_2.
	\end{split}
	\end{equation}
Plugging \eqref{Eq:disjoint faces 2} and \eqref{Eq:disjoint faces 3} into \eqref{Eq:disjoint faces}, we obtain that there exists $C_3$ such that, for $u$ large enough,
\begin{equation}\label{Eq:disjoint faces 4}
	\begin{split}
\E\{M_u (K) M_u (K')\} 
&\le C_3\sup_{t\in K, t'\in K'} \E\{(C_1+|X(t)X(t')|^{N_1}) \mathbbm{1}_{\{X(t) \geq u, X(t') \geq u\}} \}\\
&\le C_3\sup_{t\in K, t'\in K'} \E\{(C_1+(X(t)+X(t'))^{2N_1}) \mathbbm{1}_{\{[X(t)+X(t')]/2 \geq u\}} \}\\
&\le C_3 \exp\left( -\frac{u^2}{(1+\rho)\sigma_T^2} + \ep u^2\right),
\end{split}
\end{equation}
where $\ep$ is any positive number and $\rho=\sup_{t\in K, t'\in K'}{\rm Corr}[X(t), X'(t)]<1$ due to $({\bf H}2)$. The case when one of the dimensions 
of $K$ and $K'$ is zero can be proved similarly.
\end{proof}

\subsection{Adjacent faces}

The following result shows that the last two sums involving the joint moment of two adjacent faces in \eqref{Ineq:lowerbound je} are super-exponentially small.
\begin{proposition}\label{Lem:cross terms je} Let $\{X(t),\, t\in T\}$ be a centered Gaussian random field satisfying $({\bf H}1)$, $({\bf H}2)$ and $({\bf H}3)$. Then there exists $\alpha>0$ such that as $u\to \infty$,
	\begin{equation}\label{Eq:crossing terms je}
	\begin{split}
	\E\{M_u^E (K) M_u^E (K')\} &= o\left( \exp \left\{ -\frac{u^2}{2\sigma_T^2} -\alpha u^2 \right\}\right),
	\end{split}
	\end{equation}
where $K$ and $K'$ are different faces of $T$ with $d(K,K')=0$.
\end{proposition}
\begin{proof} Let $I:= \bar{K}\cap \bar{K'}$, which is nonempty since $d(K,K')=0$. To simplify notation, let us assume without loss of generality:
	\begin{equation*}
	\begin{split}
	\tau(K) &= \{1, \ldots, m, m+1, \ldots, k\},\\
	\tau(K')&= \{1, \ldots, m, k+1, \ldots, k+k'-m\},
	\end{split}
	\end{equation*}
	where $0 \leq m \leq k \leq k' \leq N$ and $k'\geq 1$. If $k=0$, we conventionally consider $\tau(K)= \emptyset$. Under this assumption, $K\in \partial_k T$, $K'\in \partial_{k'} T$, $\text{dim} (I) =m$, and all elements in $\ep(K)$ and $\ep(K')$ are 1.
	
	We first consider the case when $k\geq 1$ and $l\ge 1$. By the Kac-Rice formula,
	\begin{equation}\label{Eq:cross term je}
	\begin{split}
	&\quad \E\{M_u^E (K) M_u^E (K') \}\\
	&\leq \int_{K} dt\int_{K'} dt'\int_u^\infty dx \int_u^\infty dx'  \int_0^\infty dz_{k+1} \cdots \int_0^\infty dz_{k+k'-m} \int_0^\infty dw_{m+1} \cdots \int_0^\infty dw_{k}\\
	&\quad \E\big \{ |\text{det} \nabla^2 X_{|K}(t) ||\text{det} \nabla^2 X_{|K'}(t') | \big| X(t)=x, X(t')=x',  \nabla X_{|K}(t)=0,  X_{k+1}(t)=z_{k+1}, \\
	& \quad  \ldots, X_{k+k'-m}(t)=z_{k+k'-m},  \nabla X_{|K'}(t')=0, X_{m+1}(t')=w_{m+1}, \ldots, X_{k}(t')=w_{k} \big\}\\
	& \quad \times p_{t,t'}(x,x',0, z_{k+1}, \ldots,z_{k+k'-m}, 0, w_{m+1}, \ldots, w_{k})\\
	&:= \int \int \int_{K\times K'} A(t,t',u)\,dtdt',
	\end{split}
	\end{equation}
	where $p_{t,t'}(x,x',0, z_{k+1}, \ldots,z_{k+k'-m}, 0, w_{m+1}, \ldots, w_{k})$ is the density of the joint distribution of the variables involved in the given condition. We define 
	\begin{equation}\label{Eq:M0}
	\begin{split}
	\mathcal{M}_0:=\{t\in I :\, \nu(t)=\sigma_T^2, \, \nu_i(t)=0, \ \forall i=1,\ldots, k+k'-m\},
	\end{split}
	\end{equation}
	and consider two cases for $\mathcal{M}_0$.
	
	\textbf{Case (i): $\bm{\mathcal{M}_0 = \emptyset}$.} Under this case, since $I$ is a compact set, by the uniform continuity of conditional 
	variance, there exist constants $\ep_1, \delta_1>0$ such that
	\begin{equation}\label{Eq:super-exponentially small}
	\begin{split}
	&\sup_{t\in B(I, \delta_1), \, t'\in B'(I, \delta_1)} {\rm Var} ( X(t) |\nabla X_{|K}(t), \nabla X_{|K'} (t'))\leq \sigma_T^2 - \ep_1,
	\end{split}
	\end{equation}
	where $B(I, \delta_1)=\{t\in K: d(t, I)\le \delta_1\}$ and $B'(I, \delta_1)=\{t'\in K': d(t', I)\le \delta_1\}$. By partitioning $K\times K'$ into $B(I, \delta_1) \times B'(I, \delta_1)$ and $(K\times K')\backslash (B(I, \delta_1) \times B'(I, \delta_1))$ and applying the Kac-Rice formula, we obtain
	\begin{equation}\label{Eq:Mu3}
	\begin{split}
	&\quad \E\{M_u (K) M_u (K') \}\\
	&\leq \int_{(K\times K')\backslash (B(I, \delta_1) \times B'(I, \delta_1))} dt dt' \, p_{\nabla X_{|K}(t), \nabla X_{|K'} (t')}(0,0)\\
	&\quad \times \E \big\{ |{\rm det} \nabla^2 X_{|K}(t) | |{\rm det} \nabla^2 X_{|K'}(t') | \mathbbm{1}_{\{X(t)\geq u, X(t')\geq u\}}\big| \nabla X_{|K}(t)=0, \nabla X_{|K'}(t')=0 \big\} \\
	&+ \int_{B(I, \delta_1) \times B'(I, \delta_1)} dt dt'\, p_{\nabla X_{|K}(t), \nabla X_{|K'} (t')}(0,0)\\
	&\quad \times \E \big\{ |{\rm det} \nabla^2 X_{|K}(t) | |{\rm det} \nabla^2 X_{|K'}(t') | \mathbbm{1}_{\{X(t)\geq u, X(t')\geq u\}}\big|  \nabla X_{|K}(t)=0, \nabla X_{|K'}(t')=0 \big\} \\
	&:=I_1(u) + I_2(u).
	\end{split}
	\end{equation}
	Note that
	\begin{equation*}
	\begin{split}
	(K\times K') \backslash (B(I, \delta_1) \times B'(I, \delta_1)) &= \Big( (K\backslash B(I, \delta_1)) \times B'(I, \delta_1) \Big) \bigcup \Big( B(I, \delta_1)\times (K\backslash B(I, \delta_1)) \Big) \\
	&  \quad \bigcup  \Big( (K\backslash B(I, \delta_1))\times (K\backslash B(I, \delta_1)) \Big),
	\end{split}
	\end{equation*}
	where each product on the right hand side consists of two sets with a positive distance. It then follows from Proposition \ref{Lem:cross terms disjoint sets} that $I_1(u)$ is super-exponentially small. On the other hand, since $\mathbbm{1}_{\{X(t)\geq u, X(t')\geq u\}}\le \mathbbm{1}_{\{[X(t)+ X(t')]/2\geq u\}}$, one has
	\begin{equation}\label{Eq:I2}
		\begin{split}
			I_2(u)&\le\int_{B(I, \delta_1) \times B'(I, \delta_1)} dt dt' \int_u^\infty dx\,p_{\frac{X(t)+X(t')}{2}}(x|\nabla X_{|K}(t)=0, \nabla X_{|K'}(t')=0) \\
			&\quad \times \E \big\{ |{\rm det} \nabla^2 X_{|K}(t) | |{\rm det} \nabla^2 X_{|K'}(t') | \big| [X(t)+ X(t')]/2=x,\\
			&\qquad \nabla X_{|K}(t)=0, \nabla X_{|K'}(t')=0 \big\}p_{\nabla X_{|K}(t), \nabla X_{|K'} (t')}(0,0).
		\end{split}
	\end{equation}
	 Combining this with \eqref{Eq:super-exponentially small}, we conclude that $I_2(u)$ and hence $\E\{M_u^E (X,K) M_u^E (X,K')\}$ are super-exponentially small.

	\textbf{Case (ii): $\bm{\mathcal{M}_0 \neq \emptyset}$.} Let
	\begin{equation*}
	\begin{split}
	B(\mathcal{M}_0,\delta_2):=\{(t,t')\in K\times K' : d(t,\mathcal{M}_0)\vee d(t',\mathcal{M}_0)\le \delta_2  \},
	\end{split}
	\end{equation*}
	where $\delta_2$ is a small positive number to be specified. Note that, by the definitions of $\mathcal{M}_0$ and 
	$B(\mathcal{M}_0,\delta_2)$, there exists $\ep_2>0$ such that
	\begin{equation}\label{Eq:delta2}
	\begin{split}
	\sup_{(t,t')\in (K\times K') \backslash B(\mathcal{M}_0,\delta_2)} &{\rm Var} ( [X(t)+X(t')]/2 |\nabla X_{|K}(t), \nabla X_{|K'} (t'))\leq \sigma_T^2 - \ep_2.
	\end{split}
	\end{equation}
	Similarly to \eqref{Eq:I2}, we obtain that $\int_{(K\times K') \backslash B(\mathcal{M}_0,\delta_2)}A(t,t',u)dtdt'$ is 
	super-exponentially small. It suffices to show below that $\int _{B(\mathcal{M}_0,\delta_2)} A(t,t',u)\,dtdt'$ is super-exponentially small.
	
	Due to $({\bf H}3)$ and Proposition \ref{prop:H3}, we can choose $\delta_2$ small enough such that for all $(t,t')\in B(\mathcal{M}_0,\delta_2)$,
	\begin{equation*}
	\begin{split}
	\La_{K\cup K'}(t)&:=-\E\{X(t)\nabla^2 X_{|K\cup K'}(t)\}=-(\E\{X(t) X_{ij}(t)\})_{i,j=1,\ldots, k+k'-m}
	\end{split}
	\end{equation*}
	are positive definite. Let $\{e_1, e_2, \ldots, e_N\}$ be the standard orthonormal basis of $\R^N$. For $t\in K$ and $t'\in K'$, let $e_{t, t'}=(t'-t)/\|t'-t\|$ and $\alpha_i(t, t')= \l e_i, \La_{K\cup K'}(t)e_{t,t'}\r$. Then
	\begin{equation}\label{Eq:orthnormal basis decomp je}
	\La_{K\cup K'}(t)e_{t,t'}=\sum_{i=1}^N \l e_i, \La_{K\cup K'}(t)e_{t,t'}\r e_i = \sum_{i=1}^N  \alpha_i(t, t') e_i
	\end{equation}
	and there exists $\alpha_0 >0$ such that for all $(t,t')\in B(\mathcal{M}_0,\delta_2)$,
	\begin{equation}\label{Eq:posdef je}
	\l e_{t,t'}, \La_{K\cup K'}(t)e_{t,t'} \r \geq \alpha_0.
	\end{equation}
	Since all elements in $\ep(K)$ and $\ep(K')$ are 1, we may write
	\begin{equation*}
	\begin{split}
	t &= (t_1, \ldots, t_m, t_{m+1}, \ldots, t_k, b_{k+1}, \ldots, b_{k+k'-m}, 0, \ldots, 0),\\
	t' &= (t'_1, \ldots, t'_m, b_{m+1}, \ldots, b_k, t'_{k+1}, \ldots, t'_{k+k'-m}, 0, \ldots, 0),
	\end{split}
	\end{equation*}
	where $t_i \in (a_i, b_i)$ for $i\in \tau(K)$ and $t'_j \in (a_j, b_j)$ for $j \in \tau(K')$. Therefore,
	\begin{equation}\label{Eq:e_k je}
	\begin{split}
	\l e_i , e_{t,t'}\r &\geq 0, \quad \forall \ m+1 \leq i \leq k,\\
	\l e_i , e_{t,t'}\r &\leq 0, \quad \forall \ k+1\leq i \leq k+k'-m,\\
	\l e_i , e_{t,t'}\r &= 0, \quad \forall \ k+k'-m< i \leq N.
	\end{split}
	\end{equation}
	Let
	\begin{equation}\label{Eq:D_k je}
	\begin{split}
	D_i &= \{ (t,t')\in B(\mathcal{M}_0,\delta_2): \alpha_i (t,t')\geq \beta_i \}, \quad \text{if}\ m+1 \leq i \leq k,\\
	D_i &= \{ (t,t')\in B(\mathcal{M}_0,\delta_2): \alpha_i (t,t')\leq -\beta_i \}, \quad \text{if}\ k+1\leq i \leq k+k'-m,\\
	D_0 &= \bigg\{ (t,t')\in B(\mathcal{M}_0,\delta_2): \sum_{i=1}^m \alpha_i (t,t')\l e_i , e_{t,t'}\r\geq \beta_0\bigg \},
	\end{split}
	\end{equation}
	where $\beta_0, \beta_1,\ldots, \beta_{k+k'-m}$ are positive constants such that $\beta_0 + \sum_{i=m+1}^{k+k'-m} \beta_i < \alpha_0$.
	It follows from (\ref{Eq:e_k je}) and (\ref{Eq:D_k je}) that, if $(t,s)$ does not belong to any of $D_0, D_{m+1}, \ldots, D_{k+k'-m}$, then by (\ref{Eq:orthnormal basis decomp je}),
	$$\l\La_{K\cup K'}(t)e_{t,t'}, e_{t,t'}\r = \sum_{i=1}^N \alpha_i(t,t') \l e_i , e_{t,t'}\r\leq \beta_0 + \sum_{i=m+1}^{k+k'-m} \beta_i < \alpha_0,$$
	which contradicts (\ref{Eq:posdef je}). Thus $D_0\cup \cup_{i=m+1}^{k+k'-m} D_i$ is a covering of $B(\mathcal{M}_0,\delta_2)$. By (\ref{Eq:cross term je}),
	\begin{equation*}
	\begin{split}
	\E&\{M_u^E (K) M_u^E (K')\} \leq \int_{D_0} A(t,t',u)\,dtdt' + \sum_{i=m+1}^{k+k'-m} \int_{D_i}A(t,t',u)\,dtdt'.
	\end{split}
	\end{equation*}
	By the Kac-Rice metatheorem and the fact $\mathbbm{1}_{\{X(t)\geq u, Y(s)\geq u\}}\le \mathbbm{1}_{\{X(t)\geq u\}}$, we obtain
	\begin{equation}\label{Eq:D0}
	\begin{split}
	&\quad \int_{D_0}A(t,t',u)\,dtdt'\\
	&\leq \int_{D_0} dtdt' \int_u^\infty dx \, p_{\nabla X_{|K}(t), \nabla X_{|K'} (t')}(0,0) p_{X(t)}(x|\nabla X_{|K}(t)=0, \nabla X_{|K'} (t')=0)\\
	&\quad \times \E \big\{ |\text{det} \nabla^2 X_{|K}(t) ||\text{det} \nabla^2 X_{|K'}(t') | \big| X(t)=x, \nabla X_{|K}(t)=0, \nabla X_{|K'}(t')=0 \big\},
	\end{split}
	\end{equation}
	and that for $i= m+1, \ldots, k$,
	\begin{equation}\label{Eq:Di}
	\begin{split}
	&\quad \int_{D_i}A(t,t',u)\,dtdt'\\
	&\le \int_{D_i} dtdt' \int_u^\infty dx \int_0^\infty dw_i \, p_{X(t), \nabla X_{|K}(t), X_i(t'), \nabla X_{|K'} (t')}(x,0,w_i,0)\\
	&\quad \times \E\big \{ |\text{det} \nabla^2 X_{|K}(t)| |\text{det} \nabla^2 X_{|K'}(t')|  \big| X(t)=x, \nabla X_{|K}(t)=0, X_i(t')=w_i, \nabla X_{|K'} (t')=0 \big\}.
	\end{split}
	\end{equation}
	Comparing \eqref{Eq:D0} and \eqref{Eq:Di} with Eqs. (4.33) and (4.36) respectively in the proof of Theorem 4.8 in \citet{ChengXiao2014}, one can employ the same reasoning therein to show that ${\rm Var}(X(t)|\nabla X_{|K}(t), \nabla X_{|K'} (t'))<\sigma_T^2$ uniformly on $D_0$ and $\P(X(t)>u, X_i(t')>0|\nabla X_{|K}(t)=0, \nabla X_{|K'} (t')=0)=o(e^{-u^2/(2\sigma_T^2)-\alpha u^2})$ uniformly on $D_i$, and deduce that $\int_{D_0}A(t,t',u)\,dtdt'$ and 
	$\int_{D_i} A(t,t',u)\,dtdt'$ ($i= m+1, \ldots, k$) are super-exponentially small.
	
	It is similar to show that $\int_{D_i} A(t,t',u)\,dtdt'$ are super-exponentially small for $i=k+1, \ldots, k+k'-m$. For the 
	case $k=0$ or $l=0$, the argument is even simpler when applying the Kac-Rice formula; the details are omitted here. The proof is finished. 
\end{proof}

In the proof of Proposition \ref{Lem:cross terms je}, we have shown in \eqref{Eq:Mu3} that, 
if $\mathcal{M}_0 = \emptyset$, then the moment $\E\{M_u (X,K) M_u (X,K')\}$ is super-exponentially small. It is important to note that, the boundary condition \eqref{Eq:boundary} implies (and generalizes) the condition $\mathcal{M}_0 = \emptyset$, yielding the following result.
\begin{proposition}\label{Lem:cross terms je2} Let $\{X(t),\, t\in T\}$ be a centered Gaussian random field satisfying $({\bf H}1)$, $({\bf H}2)$ and the boundary condition \eqref{Eq:boundary}. 
	Then there exists $\alpha>0$ such that as $u\to \infty$,
	\begin{equation*}
	\begin{split}
	\E\{M_u (K) M_u (K')\} &= o\Big(\exp \Big\{ -\frac{u^2}{2\sigma_T^2} -\alpha u^2 \Big\}\Big),
	\end{split}
	\end{equation*}
	where $K$ and $K'$ are adjacent faces of $T$.
\end{proposition}


\section{Estimation of the difference between EEC and the upper bound}\label{sec:diff}
In this section, we demonstrate that the difference between $\E\{\chi(A_u)\}$ and the expected number of extended outward local maxima, i.e. the upper bound in \eqref{Ineq:upperbound je}, is super-exponentially small.
\begin{proposition}\label{Prop:simlify the high moment je} Let $\{X(t),\, t\in T\}$ be a centered Gaussian random field satisfying $({\bf H}1)$, $({\bf H}2)$ and $({\bf H}3)$. Then there exists $\alpha>0$ such that for any $K\in \partial_k T$ with $k\ge 0$, as $u\to \infty$,
	\begin{equation}\label{Eq:simplify-major-term}
	\begin{split}
	\E \{M_u^E (K)\} 
	&=(-1)^{k}\int_K \E\big\{{\rm det}\nabla^2 X_{|K}(t)\mathbbm{1}_{\{X(t)\geq u, \ \varepsilon^*_\ell X_\ell(t) \geq 0 \ {\rm for \ all}\ \ell\notin \tau(K)\}} \big| \nabla X_{|K}(t)=0\big\}\\
	&\quad \times p_{\nabla X_{|K}(t)}(0) dt +o\left( \exp \left\{ -\frac{u^2}{2\sigma_T^2} -\alpha u^2 \right\}\right)\\
	&= (-1)^{k} \E\bigg\{\bigg(\sum^k_{i=0} (-1)^i \mu_i(K)\bigg)\bigg\} + o\left( \exp \left\{ -\frac{u^2}{2\sigma_T^2} -\alpha u^2 \right\}\right).
	\end{split}
	\end{equation}
\end{proposition}

\begin{proof} The second equality in \eqref{Eq:simplify-major-term} arises from the application of the Kac-Rice formula:
	\begin{equation*}
		\begin{split}
			&\quad \E\bigg\{\bigg(\sum^k_{i=0} (-1)^i \mu_i(K)\bigg)\bigg\} \\
			&=\sum^k_{i=0} (-1)^i\int_K  \E\big\{|{\rm det}\nabla^2 X_{|K}(t)|\mathbbm{1}_{\{\text{index} (\nabla^2 X_{|K}(t))=i\}} \\
			&\quad \times  \mathbbm{1}_{\{X(t)\geq u, \ \varepsilon^*_\ell X_\ell(t) \geq 0 \ {\rm for \ all}\ \ell\notin \tau(K)\}} \big| \nabla X_{|K}(t)=0\big\}p_{\nabla X_{|K}(t)}(0)\, dt\\
			 &=\int_K  \E\big\{{\rm det}\nabla^2 X_{|K}(t)\mathbbm{1}_{\{X(t)\geq u, \ \varepsilon^*_\ell X_\ell(t) \geq 0 \ {\rm for \ all}\ \ell\notin \tau(K)\}} \big| \nabla X_{|K}(t)=0\big\} p_{\nabla X_{|K}(t)}(0) \, dt.
		\end{split}
	\end{equation*}
To prove the first approximation in \eqref{Eq:simplify-major-term} and convey the main idea, we start with the case when the face $K$ represents the interior of $T$.
	
	\textbf{Case (i):  $\bm{k=N}$.} By the Kac-Rice formula, we have
	\begin{equation*}
	\begin{split}
	&\quad \E \{M_u^E (K)\}\\
	&=\int_K p_{\nabla X(t)}(0) dt \int_u^\infty dx\, p_{X(t)}(x|\nabla X(t)=0)\E\big\{{\rm det}\nabla^2 X(t)\mathbbm{1}_{\{\nabla^2 X(t)\prec 0\}} \big| X(t)=x, \nabla X(t)=0\big\}\\
	&:=\int_K p_{\nabla X(t)}(0) dt \int_u^\infty A(t,x)dx.
	\end{split}
	\end{equation*}
	Let 
	\begin{equation}\label{Eq:O-Udelta}
	\begin{split}
	\mathcal{M}_1&=\{t\in \bar{K}=T: \nu(t)=\sigma_T^2, \ \nabla \nu(t)=2\E\{X(t)\nabla X(t)\}=0 \},\\
	B(\mathcal{M}_1, \delta_1)&=\left\{t\in K: d\left(t, \mathcal{M}_1\right)\le\delta_1\right \},
	\end{split}
	\end{equation}
	where $\delta_1$ is a small positive number to be specified. Then, we only need to estimate
	\begin{equation}\label{Eq:A1}
		\begin{split}
	\int_{B(\mathcal{M}_1, \delta_1)} p_{\nabla X(t)}(0) dt \int_u^\infty A(t,x)dx,
\end{split}
\end{equation}
	since the integral above with $B(\mathcal{M}_1, \delta_1)$ replaced by $K\backslash B(\mathcal{M}_1, \delta_1)$ 
	becomes super-exponentially small due to the fact 
	\[
	\sup_{t\in K\backslash B(\mathcal{M}_1, \delta_1)} {\rm Var}(X(t) | \nabla X(t)=0) < \sigma_T^2.
	\]
	Notice that, by Proposition \ref{prop:H3}, $\E\{X(t)\nabla^2 X(t)\} \prec 0$ for all $t\in \mathcal{M}_1$. Thus there exists $\delta_1$ small enough such that $\E\{X(t)\nabla^2 X(t)\}\prec 0$ for all $t\in B(\mathcal{M}_1, \delta_1)$. In particular, let $\la_0$ be the largest 
        eigenvalue of $\E\{X(t) \nabla^2 X(t)\}$ over $B(\mathcal{M}_1, \delta_1)$, then $\la_0<0$ by the uniform continuity. 
	Also note that $\E\{X(t)\nabla X(t)\}$ tends to 0 as $\delta_1\to 0$. Therefore, as $\delta_1\to 0$,
	\begin{equation*}
	\begin{split}
	&\quad \E\{X_{ij}(t)|X(t)=x, \nabla X(t)=0\}\\
	&=(\E\{X_{ij}(t)X(t)\}, \E\{X_{ij}(t)X_1(t)\}, \ldots, \E\{X_{ij}(t)X_N(t)\})  \left[ {\rm Cov}(X(t), \nabla X(t)) \right]^{-1} (x, 0, \ldots, 0)^T\\
	&=\frac{\E\{X_{ij}(t)X(t)\} x}{\sigma_T^2}(1+o(1)).
	\end{split}
	\end{equation*}
	 Thus, for all $x\ge u$ and $t\in B(\mathcal{M}_1, \delta_1)$ with $\delta_1$ small enough,
	\begin{equation*}
	\begin{split}
	\Sigma_1(t,x) &:=\E\{\nabla^2 X(t)|X(t)=x, \nabla X(t)=0\}\prec 0.
	\end{split}
	\end{equation*}
	Let $\Delta_1(t,x)=\nabla^2 X(t) - \Sigma_1(t,x)$.
	We have
	\begin{equation}\label{Eq:int_A}
	\begin{split}
	\int_u^\infty A(t,x)dx
	&=\int_u^\infty  p_{X(t)}(x|\nabla X(t)=0)\E\big\{{\rm det}(\Delta_1(t,x)+ \Sigma_1(t,x))\\
	&\quad \times \mathbbm{1}_{\{\Delta_1(t,x)+ \Sigma_1(t,x)\prec 0\}} \big| X(t)=x, \nabla X(t)=0\big\}\, dx\\
	&:=\int_u^\infty  p_{X(t)}(x|\nabla X(t)=0)E(t,x)\, dx.
	\end{split}
	\end{equation}
Note that the following is a centered Gaussian random matrix not depending on $x$:
\begin{equation*}
	\begin{split}
		\Omega(t)&=(\Omega_{ij}(t))_{1\le i,j\le N}=(\Delta_1(t,x) | X(t)=x,  \nabla X(t)=0).
	\end{split}
\end{equation*}
Let $h_t(v)$ denote the density of the Gaussian vector $((\Omega_{ij}^X(t))_{1\le i\le j\le N}$ with $v=(v_{ij})_{1\le i\le j\le N}\in \R^{N(N+1)/2}$. Then
\begin{equation}\label{Eq:E(t,s,xy)}
	\begin{split}
E(t,x) &= \E\big\{{\rm det}(\Omega(t)+ \Sigma_1(t,x)) \mathbbm{1}_{\{\Omega(t)+ \Sigma_1(t,x)\prec 0\}} \big\}\\
&= \int_{v:\, (v_{ij})+\Sigma_1(t,x)\prec 0} {\rm det}((v_{ij})+ \Sigma_1(t,x)) h_t(v)dv,
\end{split}
\end{equation}
where $(v_{ij})$ is the abbreviations of the matrix $v=(v_{ij})_{1\le i, j\le N}$. 
There exists a constant $c>0$ such that for $\delta_1$ small enough and all 
$t\in B(\mathcal{M}_1, \delta_1)$, and $x\ge u$, we have
\[
(v_{ij})+ \Sigma_1(t,x) \prec 0, \quad \forall \|(v_{ij})\|:=\Big(\sum_{i,j=1}^N v_{ij}^2\Big)^{1/2}<cu.
\]
This implies $\{v:\, (v_{ij})+\Sigma_1(t,x)\not\prec 0\} \subset \{v:\, \|(v_{ij})\|\ge cu\}$. Consequently, the integral in 
\eqref{Eq:E(t,s,xy)} with the domain of integration replaced by $\{v:\, (v_{ij})+\Sigma_1(t,x)\not\prec 0\}$ is $o(e^{-\alpha'u^2})$ uniformly for all $t\in B(\mathcal{M}_1, \delta_1)$, where $\alpha'$ is 
a positive constant. As a result, we conclude that, uniformly for all $t\in B(\mathcal{M}_1, \delta_1)$ and $x\ge u$,
\begin{equation*}
	\begin{split}
		E(t,x) = \int_{\R^{N(N+1)/2}} {\rm det}((v_{ij})+ \Sigma_1(t,x))h_t(v)dv + o(e^{-\alpha'u^2}).
	\end{split}
\end{equation*}
By substituting this result into \eqref{Eq:int_A}, we observe that the indicator function $\mathbbm{1}_{\{\nabla^2 X(t)\prec 0\}}$ in \eqref{Eq:A1} can be eliminated, causing only a super-exponentially small error. 
Thus, for sufficiently large $u$, there exists $\alpha>0$ such that 
	\begin{equation*}
	\begin{split}
	\E \{M_u^E (K)\} 
	&=\int_K p_{\nabla X(t)}(0) dt \int_u^\infty  \E\{{\rm det}\nabla^2 X(t)|X(t)=x, \nabla X(t)=0\} \\
	&\quad \times p_{X(t)}(x|\nabla X(t)=0)dx +o\Big( \exp \Big\{ -\frac{u^2}{2\sigma_T^2} -\alpha u^2 \Big\}\Big).
	\end{split}
	\end{equation*}

\textbf{Case (ii):  $\bm{k,l\ge 0}$.} It is worth noting that when $k=0$, the terms in \eqref{Eq:simplify-major-term} related to the Hessian will vanish, simplifying the proof.  Therefore, without loss of generality, let $k\ge 1$, $\tau(K)= \{1, \cdots, k\}$ and assume all the elements in $\ep(K)$ are 1. By the Kac-Rice formula,
	\begin{equation*}
	\begin{split}
	\E \{M_u^E (K)\}
	&=(-1)^k\int_K p_{\nabla X_{|K}(t)}(0) dt \int_u^\infty  p_{X(t)}\big(x\big |\nabla X_{|K}(t)=0\big)\E\big\{{\rm det}\nabla^2 X_{|K}(t) \\
	&\quad \times  \mathbbm{1}_{\{\nabla^2 X_{|K}(t)\prec 0\}}\mathbbm{1}_{\{X_{k+1}(t)>0,\ldots, X_N(t)>0\}} \big| X(t)=x, \nabla X_{|K}(t)=0\big\}dx\\
	&:=(-1)^k\int_K p_{\nabla X_{|K}(t)}(0) dt \int_u^\infty  A'(t,x)dx.
	\end{split}
	\end{equation*}
Let
\begin{equation}\label{Eq:M2}
	\begin{split}
		\mathcal{M}_2&=\{t\in \bar{K}: \nu(t)=\sigma_T^2, \ \nabla \nu_{|K}(t)= 2\E\{X(t)\nabla X_{|K}(t)\}=0 \},\\
		B(\mathcal{M}_2, \delta_2)&=\left\{t\in K: d\left(t, \mathcal{M}_2\right)\le\delta_2\right \},
	\end{split}
\end{equation}
where $\delta_2$ is another small positive number to be specified. Here, we only need to estimate
\begin{equation}\label{Eq:A'}
	\begin{split}
\int_{B(\mathcal{M}_2, \delta_2)} p_{\nabla X_{|K}(t)}(0) dt \int_u^\infty  A'(t,x)dx,
	\end{split}
\end{equation}
since the integral above with $B(\mathcal{M}_2, \delta_2)$ replaced by $K\backslash B(\mathcal{M}_2, \delta_2)$ 
is super-exponentially small due to the fact 
\[
\sup_{t\in K\backslash B(\mathcal{M}_2, \delta_2)} {\rm Var}(X(t) | \nabla X(t)=0) < \sigma_T^2.
\]
On the other hand, following similar arguments in the proof for Case (i), we have that removing the indicator functions 
$\mathbbm{1}_{\{\nabla^2 X_{|K}(t)\prec 0\}}$ in \eqref{Eq:A'} will  
only cause a super-exponentially small error. Combining these results, we conclude that the first approximation 
in \eqref{Eq:simplify-major-term} holds, thus completing the proof.	
\end{proof}

From the proof of Proposition \ref{Prop:simlify the high moment je}, it is evident that the same line of reasoning and arguments can be readily extended to $\E\{M_u (X,K)\}$, leading to the following result.
\begin{proposition}\label{Prop:simlify the high moment je2} Let $\{X(t),\, t\in T\}$ be a centered Gaussian random field satisfying $({\bf H}1)$, $({\bf H}2)$ and $({\bf H}3)$. Then there exists a constant $\alpha>0$ such that for any $K\in \partial_k T$, as $u\to \infty$,
	\begin{equation*}
	\begin{split}
	\E \{M_u (K)\}
	&=(-1)^k\int_K \E\big\{{\rm det}\nabla^2 X_{|K}(t)\mathbbm{1}_{\{X(t)\geq u\}}\big|\nabla X_{|K}(t)=0\big\}p_{\nabla X_{|K}(t)}(0) dt\\
	&\quad  +o\left( \exp \left\{ -\frac{u^2}{2\sigma_T^2} -\alpha u^2 \right\}\right).
	\end{split}
	\end{equation*}
\end{proposition}


\section{Proofs of the main results}\label{sec:proof}
\begin{proof}[Proof of Theorem \ref{Thm:MEC approximation je}]
	By Propositions \ref{Lem:factorial moments je}, \ref{Lem:cross terms disjoint sets} and \ref{Lem:cross terms je}, together 
	with the fact $M_u^E (K)\le M_u (K)$, we obtain that the factorial moments and the last two sums in \eqref{Ineq:lowerbound je} 
	are super-exponentially small. Therefore, from \eqref{Ineq:upperbound je} and \eqref{Ineq:lowerbound je}, it follows that there exists 
	a constant $\alpha>0$ such that as $u\to \infty$,
	\begin{equation*}
		\begin{split}
			\P\left\{\sup_{t\in T} X(t) \geq u\right\}= \sum_{k=0}^N\sum_{K\in \partial_k T} \E \{M_u^E (K) \}+ o\left( \exp \left\{ -\frac{u^2}{2\sigma_T^2} -\alpha u^2 \right\}\right).
		\end{split}
	\end{equation*}	
	This desired result follows as an immediate consequence of Proposition \ref{Prop:simlify the high moment je}.
\end{proof}

\begin{proof}[Proof of Theorem \ref{Thm:MEC approximation je2}]
	 Remark \ref{remark:M_u} indicates that both inequalities \eqref{Ineq:upperbound je} and \eqref{Ineq:lowerbound je} hold with $M_u^E(\cdot)$ 
	 replaced by $M_u(\cdot)$. Therefore, the corresponding factorial moments and the last two sums in \eqref{Ineq:lowerbound je} 
	 with $M_u^E(\cdot)$ replaced by $M_u(\cdot)$ are super-exponentially small by Propositions \ref{Lem:factorial moments je}, 
	 \ref{Lem:cross terms disjoint sets} and \ref{Lem:cross terms je2}. Consequently, there exists a constant $\alpha>0$ such that as $u\to \infty$,
	 \begin{equation*}
	 	\begin{split}
	 		\P\left\{\sup_{t\in T} X(t) \geq u\right\}
	 		= \sum_{k=0}^N\sum_{K\in \partial_k T} \E \{M_u (K)\}+ o\left( \exp \left\{ -\frac{u^2}{2\sigma_T^2} -\alpha u^2 \right\}\right).
	 	\end{split}
	 \end{equation*}	
	 The desired result follows directly from Proposition \ref{Prop:simlify the high moment je2}.
\end{proof}

\begin{proof}[Proof of Theorem \ref{Thm:MEC approximation}]
	Note that, in the proof of Theorem \ref{Thm:MEC approximation je}, we have seen that the points in $\mathcal{M}_2$ defined in 
	\eqref{Eq:M2} make major contribution to the excursion probability. That is, with up to a super-exponentially small error, 
	we can focus only on those faces, say $J$, whose closure $\bar{J}$ contains the unique point 
	$t^*$ with $\nu(t^*)=\sigma_T^2$ and satisfying $\tau(J)\subset \mathcal{I}(t^*)$
	(i.e., the partial derivatives of $\nu$ are 0 at $t^*$ restricted on $J$). To formalize this concept, we define a set of faces $T^*$ as follows:
	\begin{equation*}
		\begin{split}
			T^* &=\{J\in \partial_k T: t^*\in \bar{J}, \, \tau(J)\subset \mathcal{I}(t^*), \, k=0, \ldots, N\}.
		\end{split}
	\end{equation*}
For each $J\in T^*$, let
\begin{equation*}
	\begin{split}
		M_u^{E^*} (J) & := \# \{ t\in J: X(t)\geq u, \nabla X_{|J}(t)=0,  \nabla^2 X_{|J}(t)\prec 0, \\
		& \qquad \qquad \qquad \varepsilon^*_jX_j(t) \geq 0 \ {\rm for \ all}\ j\in \mathcal{I}(t^*)\setminus \tau(J) \}.
	\end{split}
\end{equation*}
Note that, both inequalities \eqref{Ineq:upperbound je} and \eqref{Ineq:lowerbound je} remain valid when we replace $M_u^E(J)$ with $M_u^{E^*}(J)$ for faces $J$ belonging to $T^*$, and replace $M_u^E(J)$ with $M_u(J)$ otherwise. Employing analogous reasoning as used in the derivation of  Theorems \ref{Thm:MEC approximation je} and \ref{Thm:MEC approximation je2}, 
we obtain that, there exists $\alpha>0$ such that as $u\to \infty$,
\begin{equation*}
	\begin{split}
		\P\left\{\sup_{t\in T} X(t) \geq u\right\}= \sum_{J\in T^*} \E \{M_u^{E^*} (J) \}+ o\left( \exp \left\{ -\frac{u^2}{2\sigma_T^2} -\alpha u^2 \right\}\right).
	\end{split}
\end{equation*}	
	This desired result is then deduced from Proposition \ref{Prop:simlify the high moment je}.
\end{proof}


\section{Gaussian fields with a unique maximum point of the variance}\label{sec:unique-max} 

In this section, we delve deeper into EEC approximations when the variance function $\nu(t)$ reaches its maximum value $\sigma_T^2$ at a solitary point $t^*$. While Theorem \ref{Thm:MEC approximation} provides an implicit formula for such scenarios, our objective here is to obtain explicit formulae by employing integral approximation techniques based on the Kac-Rice formula. To facilitate this process, we begin by presenting some auxiliary results related to the Laplace method for integral approximations.

\subsection{Auxiliary lemmas on Laplace approximation}
The following two formulas state the results on the Laplace approximation method. Lemma \ref{Lem:Laplace method 1}
can be found in many books on the approximations of integrals; here we refer to \citet{Wong2001}. Lemma
\ref{Lem:Laplace method 2} can be derived by following similar arguments in the proof of the Laplace method
for the case of boundary points in \cite{Wong2001}.
\begin{lemma} \label{Lem:Laplace method 1}{\rm [Laplace method for interior points]} Let $t_0$ be an interior point of $T$. Suppose the following conditions hold: (i) $g(t) \in C(T)$ and $g(t_0) \neq 0$; (ii) $h(t) \in C^2(T)$ and attains its minimum only at $t_0$; and (iii) $\nabla^2 h(t_0)$ is positive definite. Then as $u \to \infty$,
	\begin{equation*}
		\int_T g(t) e^{-uh(t)} dt =\frac{(2\pi)^{N/2}}{u^{N/2} ({\rm det} \nabla^2 h(t_0))^{1/2}}  g(t_0) e^{-uh(t_0)}  (1+ o(1)).
	\end{equation*}
\end{lemma}

\begin{lemma} \label{Lem:Laplace method 2}{\rm [Laplace method for boundary points]} Let $t_0\in K \in \partial_k T$ with $0\leq k\leq N-1$. Suppose that conditions (i), (ii) and (iii) in Lemma \ref{Lem:Laplace method 1} hold, and additionally $\nabla h(t_0)=0$. Then as $u \to \infty$,
	\begin{equation*}
		\begin{split}
			\int_T g(t) e^{-uh(t)} dt &=\frac{(2\pi)^{N/2}  \P\{Z_{i_\ell}\ep_{i_\ell}^*>0, \forall i_\ell \notin \tau(K)\}}{u^{N/2} ({\rm det} \nabla^2 h(t_0))^{1/2}}g(t_0) e^{-uh(t_0)} (1+ o(1)),
		\end{split}
	\end{equation*}
	where $(Z_{i_1}, \ldots, Z_{i_{N-k}})$ is a centered $(N-k)$-dimensional Gaussian vector with covariance matrix $(h_{i_\ell i_{\ell'}}(t_0))_{i_\ell,i_{\ell'}\notin \tau(K)}$ and $\tau(K)$ and $\ep_{i_\ell}^*$ are defined in Section \ref{sec:notation}.
\end{lemma}

\subsection{Gaussian fields satisfying the boundary condition \eqref{Eq:boundary}}
For $t\in T$, we define the following notation for conditional variances:
\begin{equation}\label{eq:La}
	\begin{split}
		\tilde{\nu}_{|K}(t) = {\rm Var}(X(t)|\nabla X_{|K}(t)=0), \quad \tilde{\nu}(t) = {\rm Var}(X(t)|\nabla X_{|K}(t)=0).
	\end{split}
\end{equation}
The following result provides explicit approximations to the excursion probabilities when the maximum of the variance is reached only at a single point and the boundary condition \eqref{Eq:boundary} is satisfied.
\begin{theorem}\label{thm:unique-max-boundary} Let $\{X(t),\, t\in T\}$ be a centered Gaussian random field satisfying $({\bf H}1)$ and $({\bf H}2)$. Suppose $\nu$ attains its maximum $\sigma_T^2$ only at $t^* \in K \in \partial_k T$, $\nu_i(t^*)\neq 0$ for all $i\notin \tau(K)$, and $\nabla^2 \nu_{|K}(t^*)\prec 0$. Then, as $u\to \infty$, 
	\begin{equation}\label{eq:interior1}
		\begin{split}
			\P\left\{\sup_{t\in T} X(t) \geq u\right\}
			&= \Psi\left(\frac{u}{\sigma_T}\right) +o\left( \exp \left\{ -\frac{u^2}{2\sigma_T^2} -\alpha u^2 \right\}\right) \text{\rm for some $\alpha>0$,}\quad \text{ if } k=0,\\
			\P \left\{\sup_{t\in T} X(t) \geq u \right\} &= \sqrt{\frac{{\rm det}(\Sigma_K(t^*))}{{\rm det}(\La_K(t^*) + \Sigma_K(t^*))}}\Psi\left(\frac{u}{\sigma_T}\right)(1+o(1)), \quad \text{ if } k\ge 1,
		\end{split}
	\end{equation}
	where $\La_K(t^*)$ and $\Sigma_K(t^*)$ are defined in \eqref{eq:Sigma}.
\end{theorem}
\begin{proof} If $k=0$, then $\nu_i(t^*)\neq 0$ for all $i\ge 1$, and hence $\mathcal{I}(t^*)=\emptyset$. The first line of \eqref{eq:interior1} follows from Theorem \ref{Thm:MEC approximation} that 
	\begin{equation*}
		\begin{split}
			\P\left\{\sup_{t\in T} X(t) \geq u\right\}
			= \P\{X(t^*)\geq u\} +o\left( \exp \left\{ -\frac{u^2}{2\sigma_T^2} -\alpha u^2 \right\}\right).
		\end{split}
	\end{equation*}
Now, let us consider the case when $k\ge 1$. Note that the assumption on partial derivatives of $\nu(t)$ implies $\mathcal{I}(t^*)=\tau(K)$. By Theorem \ref{Thm:MEC approximation}, we have
	\begin{equation}\label{eq:EP=k1}
	\begin{split}
		 \P \left\{\sup_{t\in T} X(t) \geq u \right\} = (-1)^k I(u, K) + o\left( \exp \left\{ -\frac{u^2}{2\sigma_T^2} -\alpha u^2 \right\}\right),
	\end{split}
\end{equation}
where
\begin{equation*}
	\begin{split}
		I(u,K)& =\int_K \E\big\{{\rm det}\nabla^2 X_{|K}(t)\mathbbm{1}_{\{X(t)\geq u\}} \big| \nabla X_{|K}(t)=0\big\} p_{\nabla X_{|K}(t)}(0)dt \\
		&=\int_u^\infty \int_K  \frac{{(2\pi)^{-(k+1)/2}} }{\sqrt{\tilde{\nu}_{|K}(t){\rm det}\left(\La_K(t)\right)}}   \E\big\{{\rm det}\nabla^2 X_{|K}(t)\big|X(t)=x, \nabla X_{|K}(t)=0\big\} e^{-\frac{x^2}{2\tilde{\nu}(t)}}dtdx.
	\end{split}
\end{equation*}
Applying the Laplace method in Lemma \ref{Lem:Laplace method 1} with
	\begin{equation*}
		\begin{split}
			g(t) &= \frac{1}{\sqrt{\tilde{\nu}_{|K}(t){\rm det}\left(\La_K(t)\right)}}   \E\big\{{\rm det}\nabla^2 X_{|K}(t)\big|X(t)=x, \nabla X_{|K}(t)=0\big\},\\
			h(t) &= \frac{1}{{2\tilde{\nu}_{|K}(t)}}, \quad u = x^2,
		\end{split}
	\end{equation*}
	and noting that the Hessian matrix of $1/(2\tilde{\nu}_{|K}(t))$ evaluated at $t^*$ is 
	\begin{equation}\label{eq:Theta-Hessian}
		-\frac{1}{2\tilde{\nu}_{|K}^2(t^*)} \left(\tilde{\nu}_{ij}(t^*)\right)_{i, j \in \tau(K)}= -\frac{1}{2\sigma_T^4} \nabla^2 \tilde{\nu}_{|K}(t^*)\succ 0,
	\end{equation}
we obtain
	\begin{equation}\label{eq:I(u,k)2}
		\begin{split}
			I(u,K) &= \frac{(2\sigma_T^4)^{k/2}}{\sqrt{2\pi\sigma_T^2{\rm det}\left(\La_K(t^*)\right)}\sqrt{|{\rm det}\nabla^2 \tilde{\nu}_{|K}(t^*)|}} I(u) (1+o(1)),
		\end{split}
	\end{equation}
where 
\begin{equation}\label{eq:Q}
	\begin{split}
		I(u)&=\int^\infty_u \E\big\{{\rm det} \nabla^2 X_{|K}(t^*) \big|X(t^*)=x, \nabla X_{|K}(t^*)=0 \big\} x^{-k}e^{-\frac{x^2}{2\sigma_T^2}}\, dx\\
		& ={\rm det}(\Sigma_K(t^*))\int^\infty_u \E\big\{{\rm det} (Q\nabla^2 X_{|K}(t^*)Q) \big|X(t^*)=x, \nabla X_{|K}(t^*)=0 \big\} x^{-k}e^{-\frac{x^2}{2\sigma_T^2}} dx.
	\end{split}
\end{equation}
Here, noting that $\Sigma_K(t^*)=\E\{X(t)\nabla^2 X_{|K}(t^*)\}\prec 0$ by Proposition \ref{prop:H3}, we let $Q$ in \eqref{eq:Q} be a $k \times k $ positive definite matrix such that $Q(-\Sigma_K(t^*))Q = I_k$, where $I_k$ is the size-$k$ identity matrix. Then
\begin{equation*}
	\E\{X(t)(Q \nabla^2 X_{|K}(t^*) Q)\} = Q\Sigma_K(t^*)Q = -I_k.
\end{equation*}
Notice that $\E\{X(t^*)\nabla X_{|K}(t^*)\}=0$ due to $\nu_{|K}(t^*)=0$, we have
\begin{equation*}
	\begin{split}
		\E\big\{Q \nabla^2 X_{|K}(t^*) Q \big|X(t^*)=x, \nabla X_{|K}(t^*)=0 \big\} = -\frac{x}{\sigma_T^2}I_k.
	\end{split}
\end{equation*}
One can write
\begin{equation*}
	\begin{split}
		\E\big\{{\rm det} (Q\nabla^2 X_{|K}(t^*)Q) \big|X(t^*)=x, \nabla X_{|K}(t^*)=0 \big\} = \E\{{\rm det} (\Delta(t^*) - (x/\sigma_T^2)I_k) \},
	\end{split}
\end{equation*}
where $\Delta(t^*)$ is a centered Gaussian random matrix with covariance independent of $x$. According to the Laplace expansion of determinant, $\E\{{\rm det} (\Delta(t^*) - (x/\sigma_T^2)I_k) \}$ is a polynomial in $x$ with the highest-order term being $(-1)^k\sigma_T^{-2k}x^k$. Plugging this into \eqref{eq:Q} and \eqref{eq:I(u,k)2}, we obtain
\[
I(u,K) = \frac{(-1)^k2^{k/2}|{\rm det}(\Sigma_K(t^*))|}{\sqrt{{\rm det}(\La_K(t^*))}\sqrt{|{\rm det}\left(\nabla^2 \tilde{\nu}_{|K}(t^*)\right)|}}\Psi\left(\frac{u}{\sigma_T}\right)(1+o(1)).
\]
Finally, note that 
\[
\tilde{\nu}_{|K}(t) = \E\{X(t)^2\} - \E\{X(t)\nabla X_{|K}(t)\}^T \La_K^{-1}(t)\E\{X(t)\nabla X_{|K}(t)\},
\]
we have
\begin{equation}\label{eq:nu''}
	\begin{split}
\nabla^2 \tilde{\nu}_{|K}(t^*) &= 2[\La_K(t^*) + \Sigma_K(t^*)] - 2[\La_K(t^*) + \Sigma_K(t^*)]\La_K^{-1}(t^*)[\La_K(t^*) + \Sigma_K(t^*)]\\
&= -2\Sigma_K(t^*)[I_k + \La_K^{-1}(t^*)\Sigma_K(t^*)].
\end{split}
\end{equation}
Therefore,
\[
I(u,K) = (-1)^k\sqrt{\frac{{\rm det}(\Sigma_K(t^*))}{{\rm det}(\La_K(t^*) + \Sigma_K(t^*))}}\Psi\left(\frac{u}{\sigma_T}\right)(1+o(1)),
\]
where $\Sigma_K(t^*)\prec 0$ by Proposition \ref{prop:H3} and $\La_K(t^*) + \Sigma_K(t^*)=\nabla^2 \nu_{|K}(t^*)/2\prec 0$ by assumption. Plugging this into \eqref{eq:EP=k1} yields the desired result.
\end{proof}

Now we apply Theorem \ref{thm:unique-max-boundary} to the 1D case when $T=[a, b]$. If $t^*=a$ or $t^*=b$, then it is a direct application of the first line in \eqref{eq:interior1}. If $t^* \in (a, b)$, then it follows from \eqref{eq:interior1} that
\[
\P \left\{\sup_{t\in [a,b]} X(t) \geq u \right\} = \sqrt{\frac{\E\{X(t^*)X''(t^*)\}}{{\rm Var}(X'(t^*))+\E\{X(t^*)X''(t^*)\}}}\Psi\left(\frac{u}{\sigma_T}\right)(1+o(1)).
\]

\subsection{Gaussian fields not satisfying the boundary condition \eqref{Eq:boundary}}
We consider here the other case when $\nu_i(t^*)\neq 0$ for some $i\notin \tau(K)$. For a symmetric matrix $B=(B_{ij})_{1\le i,j\le N}$, we call $(B_{ij})_{i,j\in \mathcal{I}}$ the matrix $B$ with indices restricted on $\mathcal{I}$.
\begin{theorem}\label{thm:unique-max} Let $\{X(t),\, t\in T\}$ be a centered Gaussian random field satisfying $({\bf H}1)$ and $({\bf H}2)$. Suppose $\nu$ attains its maximum $\sigma_T^2$ only at $t^* \in K \in \partial_k T$ such that $ \mathcal{I}(t^*)\setminus \tau(K)$ contains $m\ge 1$ indices and $(\nu_{ii'}(t^*))_{i,i'\in \mathcal{I}(t^*)}\prec 0$. Then, as $u\to \infty$,
	\begin{equation}\label{eq:interior2}
		\begin{split}
			&\quad \P\left\{\sup_{t\in T} X(t) \geq u\right\}\\
			&= \sum_{J} \sqrt{\frac{{\rm det}(\Sigma_J(t^*))}{{\rm det}(\La_J(t^*) + \Sigma_J(t^*))}}\P\{(Z_{J_1'}, \ldots, Z_{J_{j-k}'})\in E'(J)\}\\
			&\quad \times \P\big\{(X_{J_1}(t^*), \ldots, X_{J_{k+m-j}}(t^*))\in E(J) \big| \nabla X_{|J}(t^*)=0\big\}\Psi\left(\frac{u}{\sigma_T}\right)(1+o(1)),
		\end{split}
	\end{equation}
where the sum is taken over all faces $J$ such that $t^*\in \bar{J}$ and $\tau(J)\subset \mathcal{I}(t^*)$, $j={\rm dim}(J)$, 
\begin{equation*}
	\begin{split}
		&(J_1, \ldots, J_{k+m-j})=\mathcal{I}(t^*)\setminus \tau(J), \quad
		(J_1', \ldots, J_{j-k}')=\tau(J)\setminus \tau(K), \\
		&E(J)=\{(y_{J_1}, \ldots, y_{J_{k+m-j}})\in \R^{k+m-j}: \varepsilon^*_{J_\ell}(J) y_{J_\ell} \geq 0, \, \forall \ell= 1, \ldots, k+m-j\},\\
		&E'(J)=\{(y_{J_1'}, \ldots, y_{J_{j-k}'})\in \R^{j-k}: \varepsilon^*_{J_\ell'}(K) y_{J_\ell'} \geq 0, \, \forall \ell= 1, \ldots, j-k\},		
	\end{split}
\end{equation*}
$\varepsilon^*_{J_\ell}(J)$ and $\varepsilon^*_{J_\ell'}(K)$ are the $\ep^*$ numbers for faces $J$ and $K$ respectively, $(Z_{J_1'}, \ldots, Z_{J_{j-k}'})$ is a centered Gaussian vector having covariance matrix $\Sigma(t^*) + \Sigma(t^*)\La^{-1}(t^*)\Sigma(t^*)$ with indices restricted on $\tau(J)\setminus \tau(K)$, and $\La_J(t^*)$ and $\Sigma_J(t^*)$ are defined in \eqref{eq:Sigma}. In particular, for $k=0$, the term inside the sum in \eqref{eq:interior2} with $J=K=\{t^*\}$ is given by
\[
\P\{(X_{J_1}(t^*), \ldots, X_{J_m}(t^*))\in E(J) \}\Psi\left(\frac{u}{\sigma_T}\right).
\]
\end{theorem}

\begin{proof} We first prove the case when $k\ge 1$. By Theorem \ref{Thm:MEC approximation}, we have
	\begin{equation}\label{eq:EP=k11}
			\begin{split}
				\P\left\{\sup_{t\in T} X(t) \geq u\right\}
				&= \sum_{J}(-1)^j I(u,J) +o\left( \exp \left\{ -\frac{u^2}{2\sigma_T^2} -\alpha u^2 \right\}\right),
			\end{split}
		\end{equation}
	where $j={\rm dim}(J)$, the sum is taken over all faces $J$ such that $t^*\in \bar{J}$ and $\tau(J)\subset \mathcal{I}(t^*)$, and
	\begin{equation*}
		\begin{split}
			I(u,J)& =\int_J \E\big\{{\rm det}\nabla^2 X_{|J}(t)\mathbbm{1}_{\{X(t)\geq u\}} \mathbbm{1}_{\{\varepsilon^*_\ell X_\ell(t) \geq 0, \, \forall \ell\in \mathcal{I}(t^*)\setminus \tau(J)\}}\big| \nabla X_{|J}(t)=0\big\} p_{\nabla X_{|J}(t)}(0)dt \\
			&= \int_u^\infty \int_J  \frac{{(2\pi)^{-(j+1)/2}} }{\sqrt{\tilde{\nu}_{|J}(t){\rm det}\left(\La_J(t)\right)}}   \E\big\{{\rm det}\nabla^2 X_{|K}(t)\mathbbm{1}_{\{\varepsilon^*_\ell X_\ell(t) \geq 0, \, \forall \ell\in \mathcal{I}(t^*)\setminus \tau(J)\}}\big|X(t)=x, \\
			&\quad \nabla X_{|J}(t)=0\big\}  e^{-\frac{x^2}{2\tilde{\nu}_{|J}(t)}}dtdx.
		\end{split}
	\end{equation*}
	Applying the Laplace method in Lemma \ref{Lem:Laplace method 2} with
	\begin{equation*}
		\begin{split}
			g(t) &= \frac{1}{\sqrt{\tilde{\nu}_{|J}(t){\rm det}\left(\La_J(t)\right)}}   \E\big\{{\rm det}\nabla^2 X_{|J}(t)\mathbbm{1}_{\{\varepsilon^*_\ell X_\ell(t) \geq 0, \, \forall \ell\in \mathcal{I}(t^*)\setminus \tau(J)\}}\big|X(t)=x, \nabla X_{|J}(t)=0\big\},\\
			h(t) &= \frac{1}{{2\tilde{\nu}_{|J}(t)}}, \quad u = x^2,
		\end{split}
	\end{equation*}
	we obtain
	\begin{equation*}
		\begin{split}
			I(u,J) &= \frac{(2\sigma_T^4)^{j/2}\P\{(Z_{J_1'}, \ldots, Z_{J_{j-k}'})\in E'(J)\}}{\sqrt{2\pi\sigma_T^2{\rm det}\left(\La_J(t^*)\right)}\sqrt{|{\rm det}\nabla^2 \tilde{\nu}_{|J}(t^*)|}} I(u) (1+o(1)),
		\end{split}
	\end{equation*}
	where $(Z_{J_1'}, \ldots, Z_{J_{j-k}'})$ is a centered $(j-k)$-dimensional Gaussian vector having covariance matrix $\nabla^2 h(t^*)$ with indices restricted on $\tau(J)\setminus \tau(K)$, and
	\begin{equation}\label{eq:Q1}
		\begin{split}
			I(u)&=\int^\infty_u \E\big\{{\rm det} \nabla^2 X_{|J}(t^*) \mathbbm{1}_{\{\varepsilon^*_\ell X_\ell(t^*) \geq 0, \, \forall \ell\in \mathcal{I}(t^*)\setminus \tau(J)\}} \big|X(t^*)=x, \nabla X_{|J}(t^*)=0 \big\} x^{-j}e^{-\frac{x^2}{2\sigma_T^2}}\, dx\\
			& ={\rm det}(\Sigma_J(t^*))\int^\infty_u \int_{E(J)} \E\big\{{\rm det} (Q\nabla^2 X_{|J}(t^*)Q) \big|X(t^*)=x, \nabla X_{|J}(t^*)=0, X_{J_1}(t^*)=y_{J_1}, \\
			&\quad  \ldots, X_{J_{k+m-j}}(t^*)=y_{J_{k+m-j}} \big\} p(y_{J_1}, \ldots, y_{J_{k+m-j}}|x,0)x^{-j}e^{-\frac{x^2}{2\sigma_T^2}} dy_{J_1}\cdots dy_{J_{k+m-j}}dx.
		\end{split}
	\end{equation}
	Here $p(y_{J_1}, \ldots, y_{J_{k+m-j}}|x,0)$ is the pdf of $(X_{J_1}(t^*), \ldots, X_{J_{k+m-j}}(t^*) | X(t^*)=x, \nabla X_{|J}(t^*)=0)$, and $Q$ is a $j \times j $ positive definite matrix such that $Q(-\Sigma_J(t^*))Q = I_j$. Then, similarly to the arguments in the proof of Theorem \ref{thm:unique-max-boundary}, one can write the last expectation in \eqref{eq:Q1} as 
	\begin{equation*}
		\begin{split}
			\E\{{\rm det} (\Delta(t^*, y_{J_1}, \ldots, y_{J_{k+m-j}}) - (x/\sigma_T^2)I_k) \},
		\end{split}
	\end{equation*}
	where $\Delta(t^*, y_{J_1}, \ldots, y_{J_{k+m-j}})$ is a centered Gaussian random matrix with covariance independent of $x$, and hence the highest-order term in $x$ is $(-1)^jx^j/\sigma_T^{2j}$. Noting that $\E\{X(t^*)X_i(t^*)\}=0$ for all $i\in \mathcal{I}(t^*)$ and following similar arguments in the proof of Theorem \ref{thm:unique-max-boundary}, we obtain
	\begin{equation*}
		\begin{split}
	I(u,J) &= (-1)^j\sqrt{\frac{{\rm det}(\Sigma_J(t^*))}{{\rm det}(\La_J(t^*) + \Sigma_J(t^*))}}\P\{(Z_{J_1'}, \ldots, Z_{J_{j-k}'})\in E'(J)\}\\
	&\quad \times \P\{(X_{J_1}(t^*), \ldots, X_{J_{k+m-j}}(t^*))\in E(J) | \nabla X_{|J}(t^*)=0\}\Psi\left(\frac{u}{\sigma_T}\right)(1+o(1)),
\end{split}
\end{equation*}
	which yields the desired result together with \eqref{eq:EP=k11}. In particular, by \eqref{eq:nu''}, one can treat $(Z_{J_1'}, \ldots, Z_{J_{j-k}'})$ having covariance $\Sigma(t^*) + \Sigma(t^*)\La^{-1}(t^*)\Sigma(t^*)$ with indices restricted on $\tau(J)\setminus \tau(K)$ while not changing the probability that it falls in $E(J)$. Lastly, the case when $k=0$ can be shown similarly.
\end{proof}

Now we apply Theorem \ref{thm:unique-max} to the 1D case when $T=[a, b]$. Without loss of generality, assume $t^*=b$ and $\nu'(t^*)=0$. Then it follows from Theorem \ref{thm:unique-max} that
\begin{equation*}
	\begin{split}
&\quad \P \left\{\sup_{t\in [a,b]} X(t) \geq u \right\}\\
& = \left(\P\{X'(t^*)>0\} + \sqrt{\frac{\E\{X(t^*)X''(t^*)\}}{{\rm Var}(X'(t^*))+\E\{X(t^*)X''(t^*)\}}}\P\{Z>0\} \right)\Psi\left(\frac{u}{\sigma_T}\right)(1+o(1))\\
& = \frac{1}{2}\left(1 + \sqrt{\frac{\E\{X(t^*)X''(t^*)\}}{{\rm Var}(X'(t^*))+\E\{X(t^*)X''(t^*)\}}} \right)\Psi\left(\frac{u}{\sigma_T}\right)(1+o(1)),
\end{split}
\end{equation*}
where $Z$ is a centered Gaussian variable.

Denote by $\R_{+}^n=(0,\infty)^n$. To simplify the statement in Theorem \ref{thm:unique-max}, we present below another version with less notations on faces. 
\begin{corollary}\label{cor:unique-max} Let $\{X(t),\, t\in T\}$ be a centered Gaussian random field satisfying $({\bf H}1)$, $({\bf H}2)$ and $({\bf H}3)$. Suppose $\nu$ attains its maximum $\sigma_T^2$ only at $t^* \in K \in \partial_k T$ with $\tau(K)=\{1, \ldots, k\}$ such that $ \mathcal{I}(t^*)= \{1,\ldots, k, k+1, \ldots, k+m\}$ with $m\ge 1$ and $\left(\nu_{ii'}(t^*)\right)_{1\le i,i'\le k+m}\prec 0$. Then, as $u\to \infty$,
	\begin{equation}\label{eq:interior3}
		\begin{split}
			&\quad \P\left\{\sup_{t\in T} X(t) \geq u\right\}\\
			&= \sum_{j=k}^{k+m}\sum_{J\in \partial_j T:\, t^*\in \bar{J} } \sqrt{\frac{{\rm det}(\Sigma_J(t^*))}{{\rm det}(\La_J(t^*) + \Sigma_J(t^*))}}\P\{(Z_1, \ldots, Z_{j-k})\in \R_{+}^{j-k}\} \\
			&\quad \times \P\big\{(X_{j+1}(t^*), \ldots, X_{k+m}(t^*))\in \R_{+}^{k+m-j} \big|\nabla X_{|J}(t^*)=0\big\} \Psi\left(\frac{u}{\sigma_T}\right)(1+o(1)),
		\end{split}
	\end{equation}
	where $(Z_1, \ldots, Z_{j-k})$ is a centered Gaussian vector having covariance $\Sigma(t^*) + \Sigma(t^*)\La^{-1}(t^*)\Sigma(t^*)$ with indices restricted on $\{k+1, \ldots, j\}$, and $\La_J(t^*)$ and $\Sigma_J(t^*)$ are defined in \eqref{eq:Sigma}. In particular, for $k=0$, the term inside the sum in \eqref{eq:interior3} with $J=K=\{t^*\}$ is 
	\[
	\P\{(X_1(t^*), \ldots, X_m(t^*))\in \R_{+}^m \}\Psi\left(\frac{u}{\sigma_T}\right).
	\]
\end{corollary}

\section{Examples}\label{sec:example}
Throughout this section, we consider a centered Gaussian random field $\{X(t),\, t\in T\}$ satisfying $({\bf H}1)$, $({\bf H}2)$ and $({\bf H}3)$, where $T=[a_1, b_1]\times[a_2, b_2] \subset \R^2$.
\subsection{Examples with a unique maximum point of the variance}
Suppose $\nu(t_1,t_2)$ attains the maximum $\sigma_T^2$ only at a single point $t^*=(t_1^*,t_2^*)$; and the assumptions in Theorems \ref{thm:unique-max-boundary} or \ref{thm:unique-max} are satisfied.

\noindent \textbf{Case 1: $\bm {t^* = (b_1, b_2)}$ and $\bm {\nu_1(t^*)\nu_2(t^*)\neq 0}$.} It follows directly from Theorem \ref{thm:unique-max-boundary} that
\begin{equation*}
	\begin{split}
		\P\left\{\sup_{t\in T} X(t) \geq u\right\}
		= \Psi\left(\frac{u}{\sigma_T}\right) +o\left( \exp \left\{ -\frac{u^2}{2\sigma_T^2} -\alpha u^2 \right\}\right).
	\end{split}
\end{equation*}

\noindent \textbf{Case 2: $\bm {t^* = (b_1, b_2)}$, $\bm{\nu_1(t^*)= 0}$ and $\bm{\nu_2(t^*)\neq 0}$.} It follows from Corollary \ref{cor:unique-max} that
\begin{equation*}
	\begin{split}
		&\quad \P \left\{\sup_{t\in T} X(t) \geq u \right\}\\
		& = \left(\P\{X_1(t^*)>0\} + \sqrt{\frac{\E\{X(t^*)X_{11}(t^*)\}}{{\rm Var}(X_1(t^*))+\E\{X(t^*)X_{11}(t^*)\}}}\P\{Z>0\} \right)\Psi\left(\frac{u}{\sigma_T}\right)(1+o(1))\\
		& = \frac{1}{2}\left(1 + \sqrt{\frac{\E\{X(t^*)X_{11}(t^*)\}}{{\rm Var}(X_1(t^*))+\E\{X(t^*)X_{11}(t^*)\}}} \right)\Psi\left(\frac{u}{\sigma_T}\right)(1+o(1)),
	\end{split}
\end{equation*}
where $Z$ is a centered Gaussian variable. 

\noindent \textbf{Case 3: $\bm {t^* = (b_1, b_2)}$ and $\bm{\nu_1(t^*)= \nu_2(t^*)= 0}$.} Applying Corollary \ref{cor:unique-max} and noting the calculations in Case 2 above, we obtain
\begin{equation*}
	\begin{split}
		&\quad \P \left\{\sup_{t\in T} X(t) \geq u \right\}\\
		& = \Bigg(\P\{X_1(t^*)>0, X_2(t^*)>0\} + \frac{1}{2}\sqrt{\frac{\E\{X(t^*)X_{11}(t^*)\}}{{\rm Var}(X_1(t^*))+\E\{X(t^*)X_{11}(t^*)\}}}  \\
		&\qquad + \frac{1}{2}\sqrt{\frac{\E\{X(t^*)X_{22}(t^*)\}}{{\rm Var}(X_2(t^*))+\E\{X(t^*)X_{22}(t^*)\}}}\\
		&\qquad + \P\{Z_1>0, Z_2>0\} \sqrt{\frac{{\rm det}(\Sigma(t^*))}{{\rm det}(\La(t^*) + \Sigma(t^*))}}\Bigg)\Psi\left(\frac{u}{\sigma_T}\right)(1+o(1)),
	\end{split}
\end{equation*}
where $(Z_1, Z_2)$ is a centered Gaussian vector with covariance $\Sigma(t^*) + \Sigma(t^*)\La^{-1}(t^*)\Sigma(t^*)$. 

\noindent \textbf{Case 4: $\bm {t^* = (t_1^*, b_2)}$, where $\bm{t_1^*\in (a_1,b_1)}$ and $\bm{\nu_2(t^*)\neq 0}$.} 
It follows directly from Theorem \ref{thm:unique-max-boundary} that
\begin{equation*}
	\begin{split}
		\P\left\{\sup_{t\in T} X(t) \geq u\right\}
		= \sqrt{\frac{\E\{X(t^*)X_{11}(t^*)\}}{{\rm Var}(X_1(t^*))+\E\{X(t^*)X_{11}(t^*)\}}}\Psi\left(\frac{u}{\sigma_T}\right)(1+o(1).
	\end{split}
\end{equation*}

\noindent \textbf{Case 5: $\bm {t^* = (t_1^*, b_2)}$, where $\bm{t_1^*\in (a_1,b_1)}$ and $\bm{\nu_2(t^*)= 0}$.} Applying Corollary \ref{cor:unique-max} and noting the calculations in Case 2 above, we obtain
\begin{equation*}
	\begin{split}
		&\quad \P \left\{\sup_{t\in T} X(t) \geq u \right\}\\
		& = \frac{1}{2}\Bigg(\sqrt{\frac{\E\{X(t^*)X_{11}(t^*)\}}{{\rm Var}(X_1(t^*))+\E\{X(t^*)X_{11}(t^*)\}}}  + \sqrt{\frac{{\rm det}(\Sigma(t^*))}{{\rm det}(\La(t^*) + \Sigma(t^*))}}\Bigg)\Psi\left(\frac{u}{\sigma_T}\right)(1+o(1)).
	\end{split}
\end{equation*}

\noindent \textbf{Case 6: $\bm {a_1<t_1^*<b_1}$ and $\bm {a_2<t_2^*<b_2}$.} It follows directly from Theorem \ref{thm:unique-max-boundary} that
\begin{equation*}
	\begin{split}
		\P\left\{\sup_{t\in T} X(t) \geq u\right\}
		= \sqrt{\frac{{\rm det}(\Sigma(t^*))}{{\rm det}(\La(t^*) + \Sigma(t^*))}}\Psi\left(\frac{u}{\sigma_T}\right)(1+o(1).
	\end{split}
\end{equation*}

\subsection{Examples with the maximum of the variance achieved on a line}
Consider the Gaussian random field $X(t)$ defined as:
\[
X(t)=\xi_1\cos t_1 + \xi_1'\sin t_1 + t_2(\xi_2\cos t_2 + \xi_2'\sin t_2),
\] 
where $t=(t_1, t_2)\in T=[a_1, b_1]\times [a_2, b_2]\subset (0,2\pi)^2$, and $\xi_1, \xi_1', \xi_2, \xi_2'$ are independent standard Gaussian random variables. This is a Gaussian random field on $\R^2$ generated from the cosine field, with an additional product of $t_2$ along the vertical direction. The constraint on the parameter space within $(0, 2\pi)^2$ is imposed to prevent degeneracy in derivatives. For this field, we have $\nu(t)=1+t_2^2$, which reaches the maximum $\sigma_T^2=1+b_2^2$ on the entire real line $L:= \{(t_1, b_2): a_1\le t_1 \le b_1\}$. Furthermore, 
\[
\nu_1(t)|_{t\in L}=0, \quad \nu_2(t)|_{t\in L} = 2b_2 >0, \quad \forall t\in L.
\]
By employing similar reasoning in the proofs of Theorems \ref{Thm:MEC approximation je} and \ref{Thm:MEC approximation je2}, we see that, in the EEC approximation $\E\{\chi(A_u)\}$, all integrals (derived from the Kac-Rice formula) over faces not contained within $\bar{L}$ are super-exponentially small. Thus, there exists $\alpha>0$ such that as $u\to \infty$,
\begin{equation}\label{eq:cosine}
	\begin{split}
		\P \left\{\sup_{t\in T} X(t) \geq u \right\}
		&=\P\{X(a_1, b_2)\ge u, X_1(a_1, b_2)<0\} + \P\{X(b_1, b_2)\ge u, X_1(b_1, b_2)>0\}  \\
		&\quad + I(u) +o\left( \exp \left\{ -\frac{u^2}{2\sigma_T^2} -\alpha u^2 \right\}\right)\\
		&= \Psi\left(\frac{u}{\sqrt{1+b_2^2}}\right) + I(u) +o\left( \exp \left\{ -\frac{u^2}{2\sigma_T^2} -\alpha u^2 \right\}\right),
	\end{split}
\end{equation}
where
\begin{equation*}
	\begin{split}
		I(u)&= -\int_{a_1}^{b_1}\E\big\{X_{11}(t_1, b_2) \mathbbm{1}_{\{X(t_1, b_2)\geq u\}}\big|X_1(t_1, b_2)=0\big\}p_{X_1(t_1, b_2)}(0)dt_1.
	\end{split}
\end{equation*}
Since $X_1(t_1, b_2)=-\xi_1\sin t_1 + \xi_1'\cos t_1$ and $X_{11}(t_1, b_2)=-\xi_1\cos t_1 - \xi_1'\sin t_1$, one has
\begin{equation*}
	\begin{split}
		{\rm Cov}(X(t_1, b_2), X_1(t_1, b_2), X_{11}(t_1, b_2))&= \begin{pmatrix}
			1+b_2^2 & 0 & -1\\
			0 & 1 & 0\\
			-1 & 0 & 1
		\end{pmatrix},
	\end{split}
\end{equation*}
which does not depend on $t_1$. Particularly, $X_1(t_1, b_2)$ is independent of both $X(t_1, b_2)$ and $X_{11}(t_1, b_2)$. Thus
\begin{equation*}
	\begin{split}
		I(u)&= -\frac{b_1-a_1}{\sqrt{2\pi}}\E\big\{X_{11}(t_1, b_2) \mathbbm{1}_{\{X(t_1, b_2)\geq u\}}\big\}\\
		&= -\frac{b_1-a_1}{\sqrt{2\pi}} \int_u^\infty \E\{X_{11}(t_1, b_2) | X(t_1, b_2)=x\}\phi\left(\frac{x}{\sqrt{1+b_2^2}}\right)dx\\
		&= \frac{b_1-a_1}{\sqrt{2\pi}} \int_u^\infty \frac{x}{1+b_2^2}\phi\left(\frac{x}{\sqrt{1+b_2^2}}\right)dx\\
		&= \frac{b_1-a_1}{\sqrt{2\pi}}\phi\left(\frac{u}{\sqrt{1+b_2^2}}\right).
	\end{split}
\end{equation*}
Substituting this expression into \eqref{eq:cosine}, we arrive at the following refined approximation:
\begin{equation*}
	\begin{split}
		\P \left\{\sup_{t\in T} X(t) \geq u \right\}
		= \Psi\left(\frac{u}{\sqrt{1+b_2^2}}\right) + \frac{b_1-a_1}{\sqrt{2\pi}}\phi\left(\frac{u}{\sqrt{1+b_2^2}}\right) +o\left( \exp \left\{ -\frac{u^2}{2\sigma_T^2} -\alpha u^2 \right\}\right),
	\end{split}
\end{equation*}
which has a super-exponentially small error.

\section*{Acknowledgments}
The author acknowledges support from NSF Grants DMS-1902432 and DMS-2220523, as well as the Simons Foundation Collaboration Grant 854127.


\bibliographystyle{plainnat}

\begin{thebibliography}{36}
	
	\bibitem[Adler(2000)]{Adler00}
	Adler, R. J. (2000).
	\newblock On excursion sets, tube formulas and maxima of random fields.
	\newblock {\it Ann. Appl. Probab.} {\bf 10}, 1--74.
	
	\bibitem[Adler and Taylor(2007)]{Adler:2007}
	Adler, R. J. and Taylor, J. E. (2007).
	\newblock \emph{Random fields and geometry}.
	\newblock Springer, New York.
	
	\bibitem[Aza\"is and Delmas(2002)]{AzaisD02}
	Aza\"is, J. M. and Delmas, C. (2002).
	\newblock Asymptotic expansions for the distribution of the
	maximum of Gaussian random fields.
	\newblock {\it Extremes.} {\bf 5}, 181--212.
	
	
	\bibitem[Aza\"is and Wschebor(2009)]{AzaisW09}
	Aza\"is, J. M. and Wschebor, M. (2009).
	\newblock { \it Level Sets and Extrema of Random Processes
		and Fields}.
	\newblock John Wiley \& Sons, Hoboken, NJ.
	
	\bibitem[Cheng and Xiao (2016)]{ChengXiao2014}
	Cheng, D. and Xiao, Y. (2016).
	\newblock The mean Euler characteristic and excursion probability of Gaussian random fields with stationary increments.
	\newblock {\it Ann. Appl. Probab.} {\bf 26},  722--759.
	
	\bibitem[Piterbarg(1996)]{Piterbarg:1996}
	Piterbarg, V. I. (1996).
	\newblock \emph{Asymptotic Methods in the Theory of Gaussian Processes and
		Fields. Translations of Mathematical Monographs 148}.
	\newblock Amer. Math. Soc., Providence, RI.
	
	\bibitem[Piterbarg (1996)]{Piterbarg96}
	Piterbarg, V. I. (1996)
	\newblock Rice's method for large excursions of Gaussian random fields.
	\newblock Technical Report NO. 478, Center for Stochastic Processes, Univ. North Carolina.
	
	\bibitem{Siegmund:1995}
	Siegmund, D. and Worsley, K. J. (1995). Testing for a signal with unknown location and scale in a stationary Gaussian random field. \emph{Ann. Statist.} {\bf 23}, 608--639.
	
	\bibitem{Sun93}
	Sun, J. (1993). Tail probabilities of the maxima of Gaussian random fields. {\it Ann. Probab.} {\bf 21}, 34--71.
	
	\bibitem[Sun (2001)]{Sun:2001}
	Sun, J. (2001).
	\newblock Multiple comparisons for a large number of parameters.
	\newblock {\it Biometrical Journal} {\bf 43},  627--643.
	
	\bibitem[Taylor and Adler(2003)]{TaylorAdler03}
	Taylor, J. E. and Adler, R. J. (2003).
	\newblock  Euler characteristics for Gaussian fields on manifolds.
	\newblock {\it  Ann. Probab.} {\bf 31},  533--563.
	
	\bibitem[Taylor et al.(2005)]{TTA05}
	Taylor, J. E., Takemura, A. and Adler, R. J. (2005).
	\newblock Validity of the expected Euler characteristic heuristic.
	\newblock {\it Ann. Probab.} {\bf 33}, 1362--1396.
	
	\bibitem{TaylorW07}
	Taylor, J. E. and Worsley, K. J. (2007). Detecting sparse signals in random fields, with an application to brain mapping. {\it J. Amer. Statist. Assoc.} {\bf 102}, 913--928.
	
	\bibitem{TaylorW08}
	Taylor, J. E. and Worsley, K. J. (2008). Random fields of multivariate test statistics, with applications to shape analysis. {\it Ann. Statist.} {\bf 36}, 1--27.
	
	\bibitem[Wong(2001)]{Wong2001}
	Wong, R. (2001).
	\newblock {\it Asymptotic Approximations of Integrals}.
	\newblock SIAM, Philadelphia, PA.
	
\end{thebibliography}
\begin{small}

\end{small}

\begin{quote}
	\begin{small}
		
		\textsc{Dan Cheng}\\
		School of Mathematical and Statistical Sciences\\
		Arizona State University\\
		900 S Palm Walk\\
		Tempe, AZ 85281, USA\\
		E-mail: \texttt{cheng.stats@gmail.com}

	\end{small}
\end{quote}

\end{document}